\documentclass[reqno,11pt]{amsart}
\makeatletter
\def\section{\@startsection{section}{1}%
	\z@{.7\linespacing\@plus\linespacing}{.5\linespacing}%
	{\bfseries
		\centering
}}
\def\@secnumfont{\bfseries}
\makeatother

\usepackage{amsmath}
\usepackage{amsfonts}
\usepackage{amssymb}
\usepackage{graphicx}
\usepackage{xcolor}
\usepackage{soul}
\usepackage{ulem}
\usepackage{lineno}
\newtheorem{theorem}{Theorem}[section]

\newtheorem{corollary}[theorem]{Corollary}

\newtheorem{definition}[theorem]{Definition}

\newtheorem{lemma}[theorem]{Lemma}

\newtheorem{proposition}[theorem]{Proposition}
\newtheorem{remark}[theorem]{Remark}

\numberwithin{equation}{section}
\usepackage{lmodern}
\usepackage[colorlinks=true, allcolors=blue]{hyperref}
\colorlet{blu1}{blue!70!black}
\colorlet{blu2}{blue!50!black}
\colorlet{blu3}{blue!70!red}
\colorlet{blu4}{blue!60!green}
\colorlet{red1}{red!80}
\colorlet{red2}{red!50!black}
\colorlet{red3}{red!70!yellow}
\colorlet{red4}{red!50!yellow}
\colorlet{yel1}{yellow!50!black}
\colorlet{yel3}{yellow!20!blue}
\colorlet{gre1}{green!60!blue}
\colorlet{gre2}{green!60!black}
\colorlet{gre3}{green!40!black}
\catcode `\@=12 \textwidth16cm \textheight23cm \hoffset-1cm   \voffset-2.0cm

\begin{document}

\begin{center}
{\bf\Large The Catalan's triangle system, the Catalan's trapezoids} \vspace{0.2cm}

{\bf\Large and (q,2)--Fock space}\vspace{1.5cm}

{\bf\large Yungang Lu}\vspace{0.8cm}
	
{Department of Mathematics, University of Bari ``Aldo Moro''}\vspace{0.2cm}

{Via E. Orabuna, 4, 70125, Bari, Italy}

\end{center}\vspace{1cm}

\centerline{\bf\large Abstract}\vspace{0.2cm}
\noindent
We provide an explicit formulation for the solution to the Catalan's triangle system using Catalan's trapezoids and a specified boundary condition.  Additionally, we study this system with various boundary conditions obtained by utilizing different types of Fock spaces.
\vspace{1cm}

\section{Introduction}
\label{DCT-Intro}

In this paper, we introduce a specific system of homogenous linear equations, closely linked to {\it Catalan's trapezoids} (as introduced in \cite{Reuveni14}, to the best of the author's knowledge). The Catalan's trapezoids encompass {\it Catalan's triangle} (refer to \cite{Bailey96} and its references) and, more specifically, include the renowned {\it Catalan numbers} as particular cases.

We are motivated to study the previously mentioned system with the goal of determining the cardinality of a particular family of non--crossing pair partitions, as discussed in subsection \ref{DCTsec1-2}.

Recall that, for any $m\in \mathbb{N}^*$, the Catalan's trapezoid of order $m$  is a countable set of natural numbers denoted as $\{C_{m}(n,k)\}_{n\in \mathbb{N},\, 0\le k\le n+m-1}$. Here, for any $m\in \mathbb{N}^*$, $n\in \mathbb{N}$ and $0\le k\le n+m-1$, $C_{m}(n,k)$
represents the number of {\bf strings} composed of $n$ X's and $k$ Y's (or $n$ 1's and $k$ $-1$'s), subject to the following specific condition: in every initial {\bf segment} of the string, the number of Y's does not exceed the number of X's by more than $m$. The explicit formula for $C_{m}(n,k)$ is given as follows  (see, e.g., \cite{Reuveni14}):
\begin{equation}\label{DCT00}
C_{m}(n,k)=\begin{cases}\binom{n+k}{k}, &\text{ if } 0\le k\le m-1\\ \binom{n+k}{k}-\binom{n+k}{k-m}, &\text{ if } m\le k\le n+m-1
\end{cases}
\end{equation}
To make the concept more versatile and applicable, one prefers to extend the definition of $C_{m}(n,k)$ to cover all $(n,k)\in\mathbb{N}^2$ by adopting the convention that $C_{m}(n,k):=0$ for any such a $(n,k)\in \mathbb{N}^2$ that $k>n+m-1$.

The Catalan's trapezoid of the order $1$ is, in fact, identical to the Catalan's triangle $\{C(n,k)\}_{n\in \mathbb{N},\, 0\le k\le n}$, i.e.
\begin{equation}\label{DCT01a} C_1(n,k)=C(n,k):=\begin{cases}1, &\text{ if }k=n=0\\ \frac{n+1-k}{n+1}\binom{n+k}{k}, &\text{ if }n\in \mathbb{N}^*,\, 0\le k\le n\end{cases}
\end{equation}
More particularly, for any $n\in \mathbb{N}$, $C(n,n)$ equals to the $n$--th Catalan number, defined as $C_n:=\frac{1}{n+1} \binom{2n}{n}$.

\begin{remark}\label{shapiro} In \cite{Shapiro76}, the author introduced an alternative {\it Catalan's triangle} denoted as $\{B(n,k):\, n\in \mathbb{N},\ 0\le k\le n\}$, defined as with $B(n,k):=\frac{k}{n}\binom{2n}{n-k}$ for any $n\in \mathbb{N}$ and $0\le k\le n$. This new Catalan's triangle is closely related to the {\bf Catalan's convolution}, as discussed in \cite{LarFre2003}, \cite{Regev2012}, \cite{Shapiro76}, \cite{Tedford2011}.
\end{remark}

It was proved in \cite{Bailey96} that the Catalan's triangle $\{C(n,k)\}_{n\in\mathbb{N}^*,\,0\le k\le n}$
serves as the unique solution to the system of equations:
\begin{equation}\label{DCT01}
\begin{cases}&C(n,0)=1 \text{ for any }n\in\mathbb{N}\\
&C(n,1)=n \text{ for any }n\in\mathbb{N}^*\\
&C(n+1,n+1)=C(n+1,n) \text{ for any }n\in\mathbb{N}^*\\
&C(n+1,k)-C(n+1,k-1)=C(n,k) \text{ for any }n\ge2\text{ and }2\le k\le n\\
\end{cases}
\end{equation}

Notice that, in virtue of the first two equalities in \eqref{DCT01}, we can easily reformulate the system \eqref{DCT01} as follows:

1) The third one in \eqref{DCT01} can be replaced by the following new version:
\begin{align}\label{DCT01b2}
C(n+1,n+1)=C(n+1,n)\qquad  \text{for any }n\in\mathbb{N}
\end{align}
(i.e., the equality \eqref{DCT01b2} holds not only for $n\in\mathbb{N}^*$ but also for $n=0$) because
\[C(n+1,n+1)\Big\vert_{n=0}=C(1,1)=1=C(1,0)=C(n+1,n) \Big\vert_{n=0}
\]

2) The final equation in \eqref{DCT01} can be replaced with the following:
\begin{align}\label{DCT01a0}C(n+1,k)-C(n+1,k-1)= C(n,k),\ \text{ for all } n\in\mathbb{N}^* \text{ and }1\le k\le n
\end{align}
This replacement is justified because when the equation is evaluated for k=1, we get:
\begin{align*}
&\big[C(n+1,k)-C(n+1,k-1)\big]_{k=1}=C(n+1,1)-C(n+1,0)\\
=&n+1-1=n=C(n,k)\Big\vert_{k=1}, \ \text{ for all }  n\in\mathbb{N}^*
\end{align*}

3) \eqref{DCT01a0} is equivalent to the following:
\begin{align}\label{DCT01a1} C(n+1,m)=\sum_{h=0}^{m}C(n,h),\ \text{ for all }  n\in\mathbb{N}^* \text{ and }0\le m\le n
\end{align}
This equivalence can be established through the following argument:

$\bullet$ If \eqref{DCT01a1} holds, we obtain \eqref{DCT01a0} as follows:
\begin{align*}
&C(n+1,k)-C(n+1,k-1)=C(n+1,m)\Big\vert_{m=k} -C(n+1,m)\Big\vert_{m=k-1}\\
\overset{\eqref{DCT01a1}}= &\sum_{h=0}^{k}C(n,h) -\sum_{h=0}^{k-1}C(n,h)=C(n,k),\ \text{ for all }  n\in{\mathbb N}^* \text{ and } 1\le k\le n
\end{align*}

$\bullet$ On the other hand, \eqref{DCT01a1} is trivial for $m=0$; To see it for any $m\ge1$, we can repeatedly apply \eqref{DCT01a0} and obtain the following result:
\begin{align*}
&C(n+1,m)= C(n,m)+C(n+1,m-1)=\ldots\notag\\
=&C(n,m)+C(n,m-1) +\ldots+C(n,2)+ C(n+1,1) =\sum_{h=0}^{m}C(n,h)
\end{align*}
where, the last equality holds due to the fact that $C(n+1,1)=n+1=C(n,1)+C(n,0)$ for any $n\in\mathbb{N}^*$.

Summing up, the system \eqref{DCT01} can be reformulated as follows:
\begin{equation}\label{DCT01x}
\begin{cases}&C(n,0)=1 \text{ for any }n\in\mathbb{N}\\
&C(n,1)=n \text{ for any }n\in\mathbb{N}^*\\
&C(n+1,n+1)=C(n+1,n) \text{ for any }n\in\mathbb{N}\\
&C(n+1,k)=\sum_{h=0}^kC(n,h) \text{ for any } n\in\mathbb{N}^*\text{ and }0\le k\le n
\end{cases}
\end{equation}
Moreover, the second equation in the system \eqref{DCT01x} can be considered redundant because:

$\bullet$ the first and third equations in the system \eqref{DCT01x} confirm that $C(1,1)=C(1,0)=1$;

$\bullet$ the first and last equations in the system \eqref{DCT01x} ensure that, for any $n\in\mathbb{N}^*$,
\begin{align*}
C(n,1)&=\sum_{h=0}^1C(n-1,h)=C(n-1,0)+C(n-1,1)
=1+C(n-1,1)\\
&=2+C(n-2,1)=\ldots=n-1+C(1,1)=n
\end{align*}
That is, the system \eqref{DCT01x} is equivalent to the following system:
\begin{equation}\label{DCT01y}
\begin{cases}&C(n,0)=1 \text{ for any }n\in\mathbb{N}\\
&C(n+1,n+1)=C(n+1,n) \text{ for any }n\in\mathbb{N}\\
&C(n+1,k)=\sum_{h=0}^kC(n,h) \text{ for any }n\in\mathbb{N}^*\text{ and }0\le k\le n
\end{cases}
\end{equation}

Suggested by the well--known that the $n-$th Catalan number $C_n:=\frac{1}{n+1}\binom{2n}{n}$ represents the count of {\it non--crossing pair partitions of the set $\{1,2,\ldots,2n\}$} (for more detail, see \cite{Bo-Spe91} and related references), it is natural to expect a connection between the Catalan's triangle $\{C(n,k)\}_{0\le k\le n}$ and non--crossing pair partitions. Indeed, we will set such a connection in Theorem \ref{DCT05}. Moreover, in Section \ref{DCTsec1}, we will also

$\bullet$ reformulate the system \eqref{DCT01y};

$\bullet$ introduce the {\bf Catalan's triangle system} \eqref{DCT25a}, which is essentially a reformulation of the last equation in the system \eqref{DCT01y};

$\bullet$ set a meaningful connection between certain properties of non--crossing pair partitions and the solution of the Catalan's triangle system \eqref{DCT25a} with a specific boundary condition given in \eqref{DCT-CTSbou}.

Section \ref{DCTsec0a} is devoted to solving the Catalan's triangle system \eqref{DCT25a}, i.e., finding the explicit formulation of $x_{n,m}$'s (which verifies the Catalan's triangle system) when $\{x_{n,1}\} _{n\in\mathbb{N}^*}$ is given. In other words, it aims to express $x_{n,m}$'s in terms of $x_{n,1}$'s.  The main result, as presented in Theorem \ref{DCT25}, confirms that for any $n\ge2$ and $2\le m\le n$, $x_{n,m}$ is a homogenous linear combination of $\{x_{k,1}: k=1,\ldots, n-m+1\}$ with the coefficients $\{C_{m-1}(k-1,k+m-3): k=1,\ldots, n-m+1\}$.

In Section \ref{DCTsec(q,2)}, we utilize the $(q,2)-$Fock space (defined in Definition \ref{(q,2)-Fock}), to provide a non--trivial example of the sequence $\{x_{n,m}\} _{n\in\mathbb{N}^*,m\le n}$ that satisfies the Catalan's triangle system. Importantly, this example features a boundary condition distinct from \eqref{DCT-CTSbou}. Moreover, this particular sequence is closed connected to a subset of the set of all pair partitions, which actually lies between the complete set of pair partitions
and the set of non--crossing pair partitions.

\section{The Catalan numbers, Catalan's triangle and non--crossing pair partition}
\label{DCTsec1}

In this section, we aim to achieve the following objectives:

$\bullet$ introduce the {\it Catalan's triangle system} by reformulating the system \eqref{DCT01y};

$\bullet$ establish the relevance between the solution of the Catalan's triangle system with the specified boundary condition \eqref{DCT-CTSbou} and various properties of non--crossing pair partitions.

\subsection{Reformulation of the system \eqref{DCT01y} and the Catalan's triangle system} \label{DCTsec1-1}

We introduce the following simple modification of the Catalan's triangle:
\begin{equation}\label{DCT05a0}
S_{n,k}:=C(n-1,n-k)=\frac{k}{n}\binom{2n-k-1}{n-1}\,,\ \ \text{ for all }  n\in\mathbb{N}^*\text{ and }1\le k\le n
\end{equation}
Then,

$\bullet$ the fact $C(n,0)=1$ for any $n\in\mathbb{N}$ becomes to $S_{n,n}=1$ for any $n\in\mathbb{N}^*$;

$\bullet$ the fact $C(n,n)=C(n,n-1)$ for any $n\in\mathbb{N}^*$ becomes to $S_{n+1,1}=S_{n+1,2}$ for any $n\in\mathbb{N}^*$;

$\bullet$ the last one in the system \eqref{DCT01y} becomes to $S_{n+1,k+1}=\sum_{j=k}^{n}S_{n,j}$ for any $n\in\mathbb{N}^*$ and $k\in\{1,\ldots,n\}$.\\
Therefore, $\{S_{n,k}\}_{n\in\mathbb{N}^*,\,1\le k\le n}$ satisfies the following {\it Catalan's triangle system}:
\begin{align} \label{DCT25a}
x_{n+1,k+1}=\sum_{j=k}^{n}x_{n,j},\quad \text{for any } n\in\mathbb{N}^*\text{ and }k\in\{1,\ldots,n\}
\end{align}
with the {\it boundary conditions}:
\begin{equation}\label{DCT-CTSbou}
x_{n,n}=1,\qquad  x_{n+1,2}=x_{n+1,1}, \quad \text{ for any } n\in\mathbb{N}^*
\end{equation}
Evidently, the Catalan's triangle system is a {\it system of homogenous linear difference equations}. The solution of the system \eqref{DCT01y} is the solution of the Catalan's triangle system \eqref{DCT25a} with the specific boundary condition:
\begin{equation}\label{DCT-CTSbou1}
x_{n,n}=1,\qquad  x_{n+1,2}=x_{n+1,1}=C_n, \quad \text{ for any } n\in\mathbb{N}^*
\end{equation}

\subsection{Pair partition, non--crossing pair partition and some their properties}
\label{DCTsec1-2}

Let's start by revisiting some fundamental concepts and properties related to pair partition and non--crossing pair partition.\smallskip

For any $n\in\mathbb{N}^*$, a collection of $n$ pairs $\{(l_h,r_h)\}_{h=1}^n$ forms a {\it pair partition} of the set $\{1,2,\ldots,2n\}$ if

\begin{align}\label{DCT01b}
\big\{l_h:h\in\{1,\ldots,n\}\big\}\cup\big\{r_h: h\in\{1,\ldots,n\} \big\} =\{1,2,\ldots,2n\}
\end{align}
\begin{equation}\label{DCT01c}
l_h<r_h \text{ for any }h\in\{1,\ldots,n\}
\end{equation}
\begin{equation}\label{DCT01d}l_1<l_2<\ldots <l_n\end{equation}

The set of all pair partitions of $\{1,2,\ldots,2n\}$ is usually denoted as $PP(2n)$. For a pair partition $\{(l_h,r_h)\}_{h=1}^n\in PP(2n)$, the terms $l_h$'s (respectively, $r_h$'s) are referred to as its left (respectively, right) indices. Importantly, it is a well--known fact that
\begin{equation}\label{DCT01d1}
\big\vert PP(2n)\big\vert=(2n-1)\cdot(2n-3)\cdot\ldots \cdot 3\cdot 1=:(2n-1)!!\end{equation}
which is, in fact, the $2n$--th moment of the standard Gaussian distribution. Here and throughout, we adopt the notation of {\bf semi--factorial} $(2n-1)!!$'s as defined in \eqref{DCT01d1}.

An important subset of $PP(2n)$ is the following:
$$NCPP(2n):=\text{the set of all the non--crossing pair partitions of } \{1,2,\ldots,2n\}$$
Hereinafter, $\{(l_h,r_h)\}_{h=1}^n\in PP(2n)$ is {\it non--crossing} if the following equivalence holds for any $1\le h<k\le n$:
\begin{equation}\label{DCT01e}l_k<r_h \text{ if and only if } r_k<r_h \end{equation}
As an analogue of \eqref{DCT01d1}, one knows that
\begin{equation}\label{DCT01d2}
\big\vert NCPP(2n)\big\vert=C_n:=n-\text{th Catalan number }=\frac{1}{n+1}\binom{2n}{n}\end{equation}
This quantity is, in fact, the $2n$--th moment of the semi--circle distribution with the variance 1.

\begin{remark}\label{DCTrem1-2} For any $n\ge2$ and a pair partition $\{(l_h,r_h)\}_{h=1}^n$, it is easy to know that

1) the condition \eqref{DCT01d} (i.e., the left indices are increasing) can be replaced by its counterpart:
\begin{equation}\label{DCT01d0}r_1<r_2<\ldots <r_n
\end{equation}

2) for any $1\le h<k\le n$, \eqref{DCT01c} and \eqref{DCT01d} ensure that one can replace \eqref{DCT01e} with
\begin{equation}\label{DCT01e0}l_h<l_k<r_h \text{ if and only if } l_h<l_k<r_k<r_h \end{equation}
\end{remark} \smallskip

The above statement 2) implies that the set $NCPP(2n)$ comprises all such pair partitions $\{(l_h,r_h)\}_{h=1}^n\in PP(2n)$ that for any $1\le h< k\le n$, the {\bf interval} $(l_k,r_k)$ is either entirely contained within or to the right of the {\bf interval} $(l_h,r_h)$.

\begin{proposition}\label{DCT02} For any $n\in\mathbb{N}^*$ and $\{(l_h,r_h)\}_{h=1}^n\in PP(2n)$,
\begin{equation}\label{DCT02a}
\big\vert\{r_h:h=1,2,\ldots,n\}\cap \{l_s,l_s+1,\ldots,2n\}\big \vert \ge n-s+1,\ \text{ for any } s\in\{1,\ldots,n\}\end{equation}
In particular, $l_1=1$ and $l_n<2n$.
\end{proposition}
\begin{proof}
It follows from \eqref{DCT01c} that for any $p\in\{1,2,\ldots,2n\}$, there are more right indices than left indices among the elements of the set $\{p,p+1,\ldots,2n\}$. In other words:
\begin{align}\label{DCT02c}&\big\vert \big\{r_h:h\in\{1,\ldots,n\} \big\}\cap \big\{p,p+1,\ldots,2n\big\}\big\vert \notag\\
\ge &\big\vert \big\{l_h:h\in\{1,\ldots,n\} \big\}\cap \big\{p,p+1,\ldots,2n\big\}\big\vert,\quad\text{for any } p\in\{1,\ldots,2n\}
\end{align}
So, \eqref{DCT02a} is derived by taking $p:=l_s$ because in this case:

$\bullet$ the expression in the left--hand side of \eqref{DCT02c} becomes identical to the left--hand side of \eqref{DCT02a};

$\bullet$ \eqref{DCT01d} (i.e., the increasing property of $l_h$'s) guarantees that the set $\big\{l_h:h\in\{1,\ldots,n\} \big\}\cap \big\{l_s,l_s+1,\ldots,2n\big\}$ is equal to $\big\{l_h:h\in\{s,s+1,\ldots,n\}\big\}$, and its cardinality is $n-s+1$.

By setting $s=n$, we obtain the following inequality:
\begin{equation}\label{DCT02b}
\big\vert \{l_n,l_n+1,\ldots,2n\}\big \vert\ge  \big\vert\{r_h:h=1,2,\ldots,n\}\cap \{l_n,l_n+1,\ldots,2n\}\big \vert \overset{\eqref{DCT02a}}\ge 1
\end{equation}
and this implies that $l_n<2n$.

Finally, \eqref{DCT01c} and \eqref{DCT01d} guarantee that $l_1$ is the minimum of the set $\big\{l_h,r_h:h \in \{1,\ldots,n\}\big\}$, which, according to \eqref{DCT01b}, is equal to $\{1,\ldots,2n\}$. Therefore
$l_1=1$.  \end{proof}

\begin{corollary}\label{DCT-rem2-7} For any $n\in\mathbb{N}^*$ and $\{(l_h,r_h)\}_{h=1}^n\in PP(2n)$, \eqref{DCT02a} and \eqref{DCT02c} are equivalent.
\end{corollary}
\begin{proof}
Clearly, we only need to show the implication from \eqref{DCT02a} to \eqref{DCT02c}.

For any $p\in\{1,\ldots,2n\}$, the inequality in \eqref{DCT02c} becomes trivial when the set $\{l_h: h\in\{1,\ldots,n\} \}\cap \{p,p+1,\ldots,2n\}$ is empty. Let's prove the inequality in the case where $\{l_h: h\in\{1,\ldots,n\}\}\cap \{p,p+1,\ldots,2n\}$ is non--empty. In this case, the set $\{t:l_t\ge p\}$ is non--empty and therefore $s:=\min\{t:l_t\ge p\}$ is well--defined. Moreover,

$\bullet$ the definition $s$ as the minimum value in the set $\{t : l_t \geq p\}$ ensures that the intersection of the sets $\big\{l_h:h\in\{1,\ldots,n\} \big\}$ and $\{p,p+1,\ldots,2n\}$ equals to $\big\{l_h: h\in\{1,\ldots, n\}\big\}\cap \big\{l_s,l_s+1, \ldots,2n\big\}$. Moreover, this set, thanks to the increasing property of $l_h$, is noting else than $\big\{l_h:h\in\{s,s+1, \ldots, n\}\big\}$. As a result, the expression on the right hand side of the inequality in \eqref{DCT02c} is equal to $n-s+1$;

$\bullet$ holds the inclusion: $\{r_h:h=1,2,\ldots,n\}\cap \{p,p+1,\ldots,2n\}\supset \{r_h:h=1,2,\ldots,n\}\cap \{l_s,l_s+1,\ldots,2n\}$ because the definition of $s$ gives $l_s\ge p$;

$\bullet$ for any $t\ge s$, we observe $r_t> l_t\ge l_s$, which implies that the set $\{r_h:h=s,\ldots, n\}$ is a subset of $\{r_h:h=1,2, \ldots,n\}\cap \{l_s,l_s+1,\ldots,2n\}$.\\
Summing up, we have
\begin{align*}
&\big\vert \{r_h:h=1,2,\ldots,n\}\cap \{p,p+1,\ldots,2n\}\big\vert\\
\ge &\big\vert \{r_h:h=1,2,\ldots,n\}\cap \{l_s,l_s+1,\ldots,2n\}\big\vert\\
\ge &\big\vert\{r_h:h=s,\ldots,n\}\big\vert=n-s+1\\
=&\big\vert \{l_h:h=1,2,\ldots,n\}\cap \{p,p+1,\ldots,2n\}\big\vert
\end{align*}
\end{proof}

\begin{proposition}\label{DCT03} For any $n\in\mathbb{N}^*$ and $L_n=\{l_1,l_2,\ldots,l_n\} \subset \{1,2,\ldots,2n\}$ with the order $l_1<l_2<\ldots<l_n$, if the following analogue of \eqref{DCT02c} holds:
\begin{align}\label{DCT02d0}
&\big\vert L_n^c\cap \{p,p+1,\ldots,2n\}\big \vert \ge \big\vert L_n\cap \{p,p+1,\ldots,2n\}\big \vert,\ \text{for any } p\in\{1,2,\ldots,2n\}
\end{align}
with $L_n^c:=\{1,\ldots,2n\}\setminus L_n$, then the following affirmations are true:

1) there exists unique element of $NCPP(2n)$ such that $L_n$ is the set of its left indices;

2) there exist $\prod_{h=1}^n(2h-l_h)$ elements of $PP(2n)$ such that $L_n$ is the set of their left indices.
\end{proposition}
\begin{proof} In fact, the affirmation 1) was proved in \cite{Ac-Lu96}, and similarly the affirmation 2) was proved in \cite{Ac-Lu2022a}. \end{proof}

\begin{remark}\label{DCT-rem2-5}
Pair partition, and specifically non--crossing pair partition, can be readily extended to any set $V_n:=\{v_1,v_2,\ldots,v_{2n}\}$ with the order $v_1<\ldots<v_{2n}$ through the medium of the $v$'s indices as follows: {\bf $\{(v_{l_h},v_{r_h})\}_{h=1}^n$ is a pair partition (respectively, non--crossing pair partition) of $V_n$ if and only if $\{(l_h,r_h)\} _{h=1}^n\in PP(2n)$ (respectively, $NCPP(2n)$)}.
\end{remark}

By using this trivial generalization, one can formulate the following result.

\begin{corollary}\label{DCT-rem2-6} For any $n\in\mathbb{N}^*$, $m\in\{1,\ldots,n\}$ and $\{(l_h,r_h)\}_{h=1}^n\in NCPP(2n)$, for any $1\le h_1<\ldots<h_m\le n$, $\{(l_{h_j},r_{h_j})\}_ {j=1}^m$ is the unique element of $NCPP(\{l_{h_j},$ $r_{h_j}: j\in\{1,2, \ldots,m\} \})$ with the left indices set $\{l_{h_j}:j\in\{1,2,\ldots, m\} \}$. \end{corollary}
\begin{proof} $\{(l_{h_j},r_{h_j})\}_{j=1}^m$ is clearly a pair partition of the set $\{l_{h_j},r_{h_j}:j=1,2, \ldots,m\}$, where the left indices set is $\{l_{h_j}:j=1,2,\ldots,m\}$. The non--crossing property of this pair partition is guaranteed by the non--crossing property of $\{(l_h,r_h)\}_{h=1}^n$. \end{proof}

\subsection{Non--crossing pair partition and the Catalan's triangle system}  \label{DCTsec1-3}

In accordance with the definition of a pair partition, it is evident that, for any $n\in\mathbb{N}^*$ and $\{(l_h,r_h)\}_{h=1}^n\in PP(2n)$,

$\bullet$ $2n-l_n\ge1$ because $l_n<2n$;

$\bullet$ $2n-l_n\le n$ because $1=l_1<l_2<\ldots<l_n$ (so $l_n\ge n$).\\
Thanks to these observations, for any $n\in\mathbb{N} ^*$ and $k\in\{1,\ldots,n\}$, we reasonably introduce the following sets:
\begin{align}
PP_k(2n)&:=\big\{ \{(l_h,r_h)\}_{h=1}^n\in PP(2n):2n-l_n=k\big\} \notag\\
NCPP_k(2n)&:=\big\{ \{(l_h,r_h)\}_{h=1}^n\in NCPP(2n):2n-l_n=k\big\}
\label{DCT03a}
\end{align}
Clearly

$\bullet$ $PP(2n)=\cup_{k=1}^nPP_k(2n)$ and $PP_{k}(2n)$'s are pairwise disjoint;

$\bullet$ similarly, $NCPP(2n)=\cup_{k=1}^nNCPP_k(2n)$ and $NCPP_{k}(2n)$'s are pairwise disjoint.

\begin{proposition}\label{DCT04} For any $n\in\mathbb{N}^*$, we have the following assertions:

1) For $k=n$,
\begin{equation}\label{DCT04a0}
NCPP_n(2n)=\{(h,2n+1-h)\}_{h=1}^n;\ \ PP_n(2n)=\big\{ \{(h,r_{n+\sigma(h)})\}_{h=1}^n: \sigma\in\mathfrak{S}_n\big\}
\end{equation}
Here, $\mathfrak{S}_n$ denotes the symmetric group of the order $n$. Consequently
\begin{equation}\label{DCT04a}
\big\vert PP_n(2n)\big\vert =n!\,;\quad \big\vert NCPP_n(2n)\big\vert=1
\end{equation}

2) For $k=1$,
\begin{align}
\big\vert PP_1(2n)\big\vert &=\big\vert PP(2n-2)\big\vert =(2n-3)!!\label{DCT04b}\\
\big\vert NCPP_1(2n)\big\vert &=\big\vert NCPP(2n-2) \big\vert =\frac{1}{n}\binom{2n-2}{n-1}\label{DCT04b0}
\end{align}

3) In the case of $n>2$ and $k\in\{2,3,\ldots,n-1\}$,
\begin{equation}\label{DCT04c}
NCPP_k(2n)\text{ and }\cup_{h=k-1}^{n-1} NCPP_h(2n-2)\text{ are bijective}
\end{equation}
In particular,
\begin{equation}\label{DCT04c1}
NCPP_2(2n)\text{ and }\cup_{h=1}^{n-1}NCPP_h(2n-2)\ (\text{i.e., }NCPP(2n-2))\text{ are bijective}
\end{equation}
and
\begin{equation}\label{DCT04c2}
\big\vert NCPP_k(2n)\big\vert =
\sum_{h=k-1}^{n-1}\big\vert NCPP_h(2n-2)\big\vert
\end{equation}
More particularly,
\begin{equation}\label{DCT04c3}
\big\vert NCPP_2(2n)\big\vert =\big\vert NCPP_1(2n)\big\vert =\big\vert NCPP(2n-2)\big\vert =C_{n-1}=\frac{1}{n} \binom{2n-2}{n-1}
\end{equation}
and
\begin{equation}\label{DCT04b1}
\big\vert NCPP_{n-1}(2n)\big\vert=n-1
\end{equation}
\end{proposition}

\begin{proof} Statement 1) is trivial because when $2n-l_n=n$ (i.e. $l_n=n$), \eqref{DCT01d} necessitates $l_h=h$ for any $h\in\{1,\ldots,n\}$.

In both \eqref{DCT04b} and \eqref{DCT04b0}, the second equality is a well--known result and we clarify the first: For the case of $k=1$ (i.e., $l_n=2n-1$), $r_n$ must be $2n$ and so the application which transforms $\{(l_h,r_h)\}_{h=1}^n\in PP_n(2n)$ to $\{(l_h,r_h)\}_{h=1}^{n-1}\in PP(2n-2)$ is clearly a bijection. Additionally, Corollary \ref{DCT-rem2-6} guarantees that if $\{(l_h,r_h)\}_{h=1}^n$ is non--crossing, then $\{(l_h,r_h)\}_{h=1}^{n-1}$ is non--crossing. This discussion justifies the first equality in both \eqref{DCT04b} and \eqref{DCT04b0}.

To prove \eqref{DCT04b1}, observe that any $\{(l_h,r_h)\}_{h=1}^{n}\in NCPP_{n-1}(2n)$ must verify $l_{n}=n+1$. Consequently, all integers from $n+2$ to  $2n$ are right indices. This implies that among $\{2,3,\ldots,n\}$, there exists exactly one right index. In other words, when $l_{n}=n+1$, any $r\in\{2,3,\ldots,n\}$ uniquely associates with an element of $NCPP_{n-1}(2n)$ whose right indices are $\{r,n+2,n+3,\ldots,2n\}$. Therefore,
$\big\vert NCPP_{n-1}(2n)\big\vert=\big\vert \{2,3,\ldots,n\}\big\vert=n-1$.

Our final objective is to prove the assertion 3), excluding the formula \eqref{DCT04b1} which is already established. It is evident that our focus should primarily be on demonstrating the validity of  \eqref{DCT04c}, as the remaining equations naturally follow as its consequences.

For any $\{(l_h,r_h)\}_{h=1}^n\in NCPP_k(2n)$ with $k\in\{2,3,\ldots,n-1\}$, the definition dictates that $l_n=2n-k$, which means that there are $k$'s right indices to the right of $l_n$. Additionally, the number of the right indices between $l_{n-1}$ and $l_n$ (i.e., $l_n-1-l_{n-1}$) can vary form 0 to $n-k$ because

$\bullet$ the increasing property of $l_h$'s implies that $l_n-1-l_{n-1}\ge0$;

$\bullet$ the inequality $l_n-1-l_{n-1}\le n-k$ holds due to the following:
\begin{align*}n=\vert\{\text{right indices}\}\vert&\ge \vert\{\text{right indices to the right of }l_n\}\vert\\
&+\vert\{\text{right indices between $l_{n-1}$ and }l_n \} \vert\\
&=2n-k+l_n-1-l_{n-1}
\end{align*}

Let, for any $m\in\{0,1,\ldots,n-k\}$
\begin{equation}\label{DCT04e}
NCPP_{k,m}(2n):=\big\{ \{(l_h,r_h)\}_{h=1}^n\in NCPP_k(2n):l_{n-1}=l_n-1-m \big\}
\end{equation}
Then

$\bullet$ $NCPP_{k}(2n)=\cup_{m=0}^{n-k}NCPP_{k,m}(2n)$;

$\bullet$ the set $NCPP_{k,m}(2n)$'s, for $m\in\{0,1,\ldots,n-k\}$, are pairwise disjoint.

We introduce the application $\pi: \{(l_h,r_h)\}_{h=1}^n\longmapsto \{(l_h,r_h)\}_{h=1}^{n-1}$. According to Corollary \ref{DCT-rem2-6}, $\{(l_h,r_h)\}_{h=1}^{n-1}$ forms a non--crossing pair partition of the set $\{1,\ldots, 2n\}\setminus \{l_n,l_n+1\}$ (where, the non--crossing property of $\{(l_h,r_h)\}_{h=1}^n$ implies that $r_n=l_n+1$). In more detail,

$\bullet$ $\pi$ is clearly a bijection;

$\bullet$ as a pair partition of the set $\{1,\ldots,2n\} \setminus \{l_n,l_n+1\}$,  there are $2n-l_n-1+m$ right indices to the right of the last left index $l_{n-1}$, where, $2n-l_n-1+m=2(n-1)-l_{n-1}$ since $m=l_n-l_{n-1}-1$.\\
So, $NCPP_{k,m}(2n)$ is a bijective correspondence with
$NCPP_{2(n-1)-l_{n-1}} (2(n-1))$ and
$2(n-1)-l_{n-1}=2n-l_n-1+m=k+m-1$. Summing up,
$NCPP_{k}(2n)=\cup_{m=0}^{n-k}NCPP_{k,m}(2n)$ is in a bijective
correspondence with  $\cup_{m=0}^{n-k}NCPP_{k+m-1}(2(n-1))$,
which, by denoting $h:=m+k-1$, is equal to $\cup_{h=k-1}
^{n-1}NCPP_{h}(2(n-1))$.  \end{proof}

The following result sets a connection between Catalan's triangle and non--crossing pair partitions.

\begin{theorem}\label{DCT05} For any $n\in\mathbb{N}^*$ and $k\in\{1,\ldots,n\}$,
\begin{equation}\label{DCT05a}
\big\vert NCPP_{k}(2n)\big\vert=C(n-1,n-k)=\frac{k}{n} \binom{2n-k-1}{n-1}
\end{equation}
\end{theorem}
\begin{remark}From \eqref{DCT04c3}, we have
\[
\frac{k}{n}\binom{2n-k-1}{n-1}\bigg\vert_{k=1}=\frac{1}{n}
\binom{2n-2} {n-1}=\frac{k}{n} \binom{2n-k-1}{n-1}\bigg\vert_{k=2}   \]
In fact, the first equality is trivial and second equality holds because
\begin{align*}
\frac{k}{n}\binom{2n-k-1}{n-1}\bigg\vert_{k=2}=\frac{1}{n}
\cdot\frac{2(n-1)(2n-3)!}{ (n-1)!(n-1)!} =\frac{1}{n}\binom{2n-2}{n-1}
\end{align*}
Moreover, with the help of \eqref{DCT05a}, one gets the following well--known equality
\begin{align*}C_n\overset{\eqref{DCT01d2}}=\big\vert NCPP(2n) \big\vert=&\sum_{k=1}^n\big\vert NCPP_{k}(2n)\big\vert\\
\overset{\eqref{DCT05a}}=&\sum_{k=1}^nC(n-1,n-k)=
\sum_{h=0}^{n-1}C(n-1,h)
\end{align*}
\end{remark}

\begin{proof}[{Proof of Theorem \ref{DCT05}}] The second equality in \eqref{DCT05a} is a consequence of \eqref{DCT01a} and we see the first.

Let
\begin{equation}\label{DCT05b}
S(n,n+1-k):=\big\vert NCPP_{k}(2(n+1))\big\vert \,,\quad \text{for any } n\in\mathbb{N}\text{ and }1\le k\le n+1
\end{equation}
Then,
\begin{align}
S(n,0)&=S(n,n+1-k)\Big\vert_{k=n+1}=\big\vert NCPP_{n+1}(2(n+1)))\big\vert
\overset{\eqref{DCT04a}}{=}1,\quad\text{for any } n\in\mathbb{N} \label{DCT05b1}
\end{align}

\begin{equation}\label{DCT05b2}
S(n,1)=S(n,n+1-k)\Big\vert_{k=n}=\big\vert NCPP_{n}(2(n+1)) \big\vert\overset{\eqref{DCT04b1}}{=}n,\quad\text{for any } n\in\mathbb{N}^*
\end{equation}
and
\begin{align}
&S(n+1,n+1)=S(n+1,n+2-k)\Big\vert_{k=1}=\big\vert NCPP_{1}(2(n+2))\big\vert\notag\\
\overset {\eqref{DCT04b0}} {=}&\big\vert NCPP(2(n+1))\big\vert
\overset{\eqref{DCT04c3}} {=}\big\vert NCPP_{2}(2(n+2))\big\vert\notag\\
=&S(n+1,n+2-k)\Big \vert_{k=2} =S(n+1,n)\,,\quad\text{for any } n\in\mathbb{N}^*    \label{DCT05b3}
\end{align}
Moreover, we obtain
\begin{equation}\label{DCT05c}
S(n+1,k)-S(n+1,k-1)=\big\vert NCPP_{n+1-k}(2(n+1))\big\vert =S(n,k)
\end{equation}
because for any $1\le k\le n+1$
\begin{align}\label{DCT05c1}
&S(n+1,k)=S(n+1,n+2-(n+2-k)) \notag\\
=&\big\vert NCPP_{n+2-k}(2(n+2)) \big\vert\overset{\eqref{DCT04c2}}{=} \sum_{h=n+1-k}^{n+1}\big\vert NCPP_h(2(n+1))\big\vert
\end{align}
and
\begin{align}\label{DCT05c2}
&S(n+1,k-1)=S(n+1,n+2-(n+3-k)) \notag\\
=&\big\vert NCPP_{n+3-k}(2(n+2))\big\vert\overset {\eqref{DCT04c2}}{=} \sum_{h=n+2-k}^{n+1} \big\vert NCPP_h(2(n+1))\big\vert
\end{align}

In summary, $\{S(n,k)\}_{0\le k\le n}$ satisfies the system \eqref{DCT01}, and the uniqueness of its solution leads to $S(n,k)=C(n,k)$ for any $n\in \mathbb{N}$ and $0\le k\le n$. Therefore, \eqref{DCT05b} and the second equality in \eqref{DCT05a} confirm the thesis. \end{proof}

\section{ Solving the Catalan's triangle system \eqref{DCT25a}}\label{DCTsec0a} 

This section is devoted to solving the Catalan's triangle system \eqref{DCT25a}. The proof is divided into several propositions.

\begin{theorem}\label{DCT25} Let $\{x_{n,m}\}_{n\in \mathbb{N}^*,\,1\le m\le n}\subset\mathbb{C}$. The following statements are equivalent:

$\bullet$ $\{x_{n,m}\}_{n\in \mathbb{N}^*,\,1\le m\le n}$ satisfies the Catalan's triangle system \eqref{DCT25a};

$\bullet$ $\{x_{n,m}\}_{n\in \mathbb{N}^*,\,1\le m\le n}$ is determined by $\{x_{n,1}\}_{n\in \mathbb{N}^*}$ using the following formula:
\begin{align}\label{DCT25b}
x_{n,m}=\sum_{h=0}^{n-m}C_{m-1}(h,h+m-2)x_{n-m-h+1,1},\ \text{for any } n\ge2 \text{ and } 2\le m\le n
\end{align}
Hereinafter, $\{C_{m}(n,k)\} _{n\in \mathbb{N},\, 0\le k\le n+m-1}$ is the {\it Catalan's trapezoid of the order $m$}.
\end{theorem}

To prove this theorem, let's begin by examining  some elementary properties of the Catalan's trapezoids, which will be useful in the proof of Theorem \ref{DCT25}.

\begin{lemma}\label{DCT-le01} The Catalan's trapezoids $\big\{\{C_{m}(n,k)\} _{n\in \mathbb{N},\, 0\le k\le n+m-1}:m\in\mathbb{N}^*\big\}$ exhibit the following properties:
\begin{equation}\label{DCT23d0} C_m(r,0)=C_m(0,h)=1,\quad \text{for any } m\in \mathbb{N}^*,\ r\in   \mathbb{N}\text{ and }0\le h\le m-1
\end{equation}
\begin{equation}\label{DCT23d0a}
C_{m}(1,k):=\begin{cases}k+1, &\text{ if } 0\le k\le m-1\\ k, &\text{ if } k=m
\end{cases}\,,\quad\text{for any } m\in \mathbb{N}^*
\end{equation}
\begin{align}\label{DCT23g6}
C_{1}(n+1,n+1)=C_{1}(n+1,n)=C_{n+1}=C_{2}(n,n+1),\quad\text{for any } n\in\mathbb{N}
\end{align}
and
\begin{align}\label{DCT23g9}
&C_{m}(k,k+m-1)+C_{m-2}(k+1,k+m-2)\notag\\
=&C_{m-1} (k+1,k+m-1),\quad\text{for any } m\ge 3\text{ and } k\ge 0
\end{align}
\end{lemma}
\begin{proof}
\eqref{DCT23d0} and \eqref{DCT23d0a} are obviously direct consequences of \eqref{DCT00}.

The first two equalities in \eqref{DCT23g6} are trivial due to \eqref{DCT01} and the fact that $C_1(n,k)=C(n,k)$ for all $n\in \mathbb{N}$ and $0\le k\le n$; the third equality in \eqref{DCT23g6} can be easily proven as follows:

$\bullet$ For $n=0$, one has $C_{n+1}\Big\vert_{n=0}=C_1=1 \overset{\eqref{DCT23d0}}= C_{2}(n,n+1)\Big\vert_{n=0}$.

$\bullet$ For any $n\in\mathbb{N}^*$ (so $n+1\ge 2$), one has
\begin{align}\label{DCT23g7}
&C_{2}(n,n+1)\overset{\eqref{DCT00}}=
\binom{2n+1}{n+1}-\binom{2n+1}{n-1}
=\frac{(2n+2)!}{(n+1)!(n+2)!}=C_{n+1}
\end{align}

To prove \eqref{DCT23g9}, we start with the case of $k=0$. \eqref{DCT23d0} and \eqref{DCT23d0a} ensure that
\begin{align*}
&\Big[C_{m}(k,k+m-1)+C_{m-2}(k+1,k+m-2)\Big]_{k=0}\\
=&1+m-2=C_{m-1}(1,m-1)=C_{m-1} (k+1,k+m-1)\Big\vert_{k=0}
\end{align*}
For $k\ge1$, one has $k+m-1\ge m$ and $k+m-2\ge m-1 > m-2$. Moreover, the assumption $m\ge3$ gives $m-1> m-2\ge1$ and so \eqref{DCT00} says
\begin{align*}C_{m}(k,k+m-1)&=\binom{2k+m-1}{k}-
\binom{2k+m-1}{k-1}\\
C_{m-2}(k+1,k+m-2)&=\binom{2k+m-1}{k+1}-
\binom{2k+m-1}{k}
\end{align*}
Therefore, by taking the difference of above two terms, we find
that
\begin{align*}
&C_{m}(k,k+m-1)+C_{m-2}(k+1,k+m-2)= \binom{2k+m-1}{k+1}-
\binom{2k+m-1}{k-1}\\
=&\frac{(2k+m-1)!}{(k+1)!(k+m)!}\Big( (m+k-1)(m+k) -k(k+1)\Big)=\frac{(2k+m)!(m-1)}{(k+1)!(k+m)!}\\
=&\binom{2k+m}{k+1}-\binom{2k+m}{k}=C_{m-1} (k+1,k+m-1)
\end{align*}      \end{proof}  

The first step in proving Theorem \ref{DCT25} is to find a solution for $x_{n,n}$'s.

\begin{proposition}\label{DCT25c}Any solution of the system \eqref{DCT25a} satisfies the following equalities:
\begin{align}\label{DCT25c1}
x_{n,n}=x_{1,1}=\sum_{h=0}^{n-m} C_{m-1}(h,h+m-2) x_{n-m-h+1,1}\Big\vert_{m=n},\quad\text{for any } n\ge2
\end{align}
\end{proposition}
\begin{proof}
\eqref{DCT25a} trivially yields the first equality in \eqref{DCT25c1}:
\[x_{n,n}=\sum_{k=n-1}^{n-1}x_{n-1,k} =x_{n-1,n-1}=\ldots=x_{1,1}
\]
The second equality in \eqref{DCT25c1} is also straightforward:
\[\sum_{h=0}^{n-m} C_{m-1}(h,h+m-2)x_{n-m-h+1,1} \Big\vert_{m=n}=C_{m-1}(0,m-2)x_{1,1} \overset{\eqref{DCT23d0}}{=}x_{1,1}
\]       \end{proof}

As the second step of the proof of Theorem \ref{DCT25}, we set the uniqueness of solution for the system \eqref{DCT25a}: {\it This system has at most one solution whenever $\{x_{n,1}\}_{n\in\mathbb {N}^*}$ is given}.

\begin{proposition}\label{DCT25d} In the case where $x_{n,1}=0$ for any $n\in\mathbb{N}^*$,  the system \eqref{DCT25a} only has the trivial solution $x_{n,m}=0$ for any $n\in \mathbb{N}^*$ and $1\le m\le n$.
Moreover the solution to the system \eqref{DCT25a} is {\bf unique up to the boundary condition}: for any solutions $\{y_{n,m}\}_{n\in \mathbb{N}^*,\,1\le m\le n}$ and $\{z_{n,m}\}_{n\in \mathbb{N}^*,\,1\le m\le n}$ of the system \eqref{DCT25a},  $y_{n,m}=z_{n,m}$ for all $n\in\mathbb{N}^*$ and  $1\le m\le n$ if and only if $y_{n,1}=z_{n,1}$ for all $n\in\mathbb{N}^*$.
\end{proposition}
\begin{proof} It is sufficient to demonstrate $y_{n,m}=z_{n,m}$ for all $n\in\mathbb{N}^*$ and $1\le m\le n$ by assuming that $y_{n,1}=z_{n,1}$ for all $n\in\mathbb{N}^*$.

The assumption gives $x_{2,1}=x_{1,1}=0$, and \eqref{DCT25c1} gives $x_{2,2}=x_{1,1}=0$. Therefore,  $x_{n,k}=0$ for any $n\in\{1,2\}$ and $k\in\{1,\ldots,n\}$.

Since $x_{n,1}=0$ for any $n\in\mathbb{N}^*$, assuming $x_{n,k}=0$ for all $2\le k\le n$ is the same as to assume
\begin{align}\label{DCT23f0}
x_{n,k}=0,\quad \text{for any } k\in\{1,\ldots,n\}
\end{align}
Therefore,
\begin{align}\label{DCT23f1}
x_{n+1,1}{=}0,\quad
x_{n+1,m+1}=\sum_{k=m}^{n}x_{n,k}\overset{\eqref{DCT23f0}} {=}0,\quad \text{for any } 1\le m\le n
\end{align}
By induction, we conclude that $x_{n,m}=0$ for any $n\in \mathbb{N}^*$ and $1\le m\le n$.

If $\{y_{n,m}\}_{n\in\mathbb{N}^*,\,1\le m\le n}$ and $\{z_{n,m}\}_{n\in\mathbb{N}^*,\,1\le m\le n}$ both solve the system \eqref{DCT25a} and $y_{n,1}=z_{n,1}$ for all $n\in \mathbb{N}^*$, let's denote $u_{n,m}:= z_{n,m}-y_{n,m}$ for all $n\in\mathbb{N}^*$ and $1\le m \le n$. We find that $\{u_{n,m}\}_{n\in \mathbb{N}^*,\, 1\le m\le n}$ also satisfies the system \eqref{DCT25a} and $u_{n,1}=0$ for all $n\in\mathbb{N}^*$. Therefore, we can conclude that $u_{n,m}=0$ (i.e., $z_{n,m}= y_{n,m}$) for all $n\in\mathbb{N}^*$ and $1\le m\le n$.  \end{proof}   

From now on, we will begin solving the system \eqref{DCT25a}. Initially, in the following Proposition \ref{DCT22} and Corollary \ref{DCT25f}, we will partially solve the system \eqref{DCT25a} to find $\{x_{n,2}\}_{n\in \mathbb{N}^*}$ and $\{x_{n,3}\}_{n\in\mathbb{N}^*}$. Subsequently, in Proposition \ref{DCT23}, we will provide
generic $\{x_{n,m}\}_{n\in\mathbb{N}^*}$ for all $m\in \{2,\ldots,n\}$.

\begin{proposition}\label{DCT22} Let $\{x_{n,m}\}_{n\in\mathbb{N}^*,\, 1\le m\le n}$ be a solution of the system \eqref{DCT25a}. Then   for any $n\ge3$, we can express $x_{n,2}$ as follows:
\begin{align}\label{DCT22a}
x_{n,2}&=\sum_{h=1}^m C(m-1,m-h)\sum_{2\le k_{m+1-h}\le\ldots \le k_m\le n-m} x_{n-m,k_m}\notag\\
&+\sum_{k=0}^{m-1}C_kx_{n-k-1,1},
\qquad\text{for any }  1\le m\le n-2
\end{align}
where, $\{C(n,k)\}_{n\in\mathbb{N},\,0\le k\le n}$ represents the Catalan's triangle, and $C_k$'s are the Catalan numbers. In particular,
\begin{align}\label{DCT22b}
x_{n,2}=\sum_{k=0}^{n-2}C_kx_{n-k-1,1},\quad\text{for any } n\ge3
\end{align}
\end{proposition}
\begin{proof} After establishing \eqref{DCT22a}, we can readily obtain \eqref{DCT22b} as follows: By substituting $m=n-2$ into equation \eqref{DCT22a}, we get:
\begin{align}\label{DCT22d0}
x_{n,2}&=\sum_{k=0}^{n-3}C_kx_{n-k-1,1}\notag\\
&+\sum_{h=1}^{n-2} C(n-3,n-2-h)\sum_{2\le k_{n-1-h}\le\ldots \le k_{n-2}\le n-(n-2)} x_{n-(n-2),k_{n-2}}\notag\\
&\overset{\eqref{DCT25c1}}=\sum_{k=0}^{n-3}C_kx_{n-k-1,1} +x_{2,2}\sum_{h=1}^{n-2} C(n-3,n-2-h)
\end{align}
By comparing \eqref{DCT22b} and \eqref{DCT22d0}, we can conclude that \eqref{DCT22b} is equivalent to
\begin{align}\label{DCT22d5}
x_{2,2}\sum_{h=1}^{n-2} C(n-3,n-2-h)=x_{1,1}C_{n-2},\quad \text{for any } n\ge3
\end{align}

Thanks to the well--known properties of Catalan's triangles:
\begin{align}\label{DCT22d1}
&C(0,0)=1,\quad C(m,m)=C(m,m-1)=C_m\notag\\
&\sum_{p=0}^{h}C(m-1,p) =C(m,h),\ \text{for any } m\ge 1,\ 0\le h\le m-1
\end{align}
it is concluded that
\begin{align}\label{DCT22d2}\sum_{h=1}^{n-2} C(n-3,n-2-h)\overset{p:=n-2-h}{=} \sum_{p=0}^{n-3} C(n-3,p)=C(n-2,n-3)=C_{n-2}
\end{align}
As a consequence, we have proven \eqref{DCT22d5}:
\[x_{2,2}\sum_{h=1}^{n-2} C(n-3,n-2-h) \overset{\eqref{DCT25c1}}=x_{1,1}\sum_{h=1}^{n-2} C(n-3,n-2-h)\overset{\eqref{DCT22d2}}=x_{1,1}C_{n-2}
\]

Now, let's direct our attention to the proof of \eqref{DCT22a}. By introducing
\begin{align}
D_1(n,m)&:=\sum_{h=1}^m C(m-1,m-h)\sum_{2\le k_{m+1-h}\le\ldots \le k_m\le n-m} x_{n-m,k_m}\notag\\
D_2(n,m)&:=\sum_{k=0}^{m-1}C_kx_{n-k-1,1}\,,\qquad\quad \quad\text{for any } n\ge3,\ 1\le m\le n-2 \label{DCT22c}
\end{align}
\eqref{DCT22a} can be expressed as
\begin{align}\label{DCT22d}x_{n,2}=D_1(n,m) +D_2(n,m), \quad \text{for any } n\ge3,\ 1\le m\le n-2
\end{align}
Now, let's proceed to prove \eqref{DCT22d}. We start with the following equation:
\begin{align*}
x_{n,2}\overset{\eqref{DCT25a}}=\sum_{k=1}^{n-1}
x_{n-1,k}=x_{n-1,1}+\sum_{k=2}^{n-1}x_{n-1,k}
\end{align*}
where
\[x_{n-1,1}= C_0x_{n-1,1}\overset{\eqref{DCT25b}}=\sum_{k=0}^{m-1} C_kx_{n-k-1,1}\Big\vert_{m=1}=D_2(n,1)\]
and
\begin{align*}
&\sum_{k=2}^{n-1}x_{n-1,k} =C(0,0)\sum_{k=2}^{n-1}x_{n-1,k}\\
=& \sum_{h=1}^m C(m-1,m-h)\sum_{2\le k_{m+1-h}\le\ldots \le k_m\le n-m} x_{n-m,k_m}\Big\vert_{m=1}=D_1(n,1)
\end{align*}
These three formulae guarantee the equality in \eqref{DCT22d} for $m=1$, specifically,
\begin{align}\label{DCT22e}
x_{n,2}=D_1(n,1)+D_2(n,1)
\end{align}
So, by demonstrating, for any $1\le m\le n-3$,
\begin{align}\label{DCT22f}
D_1(n,m)+D_2(n,m)=D_1(n,m+1)+D_2(n,m+1) 
\end{align}
we can successfully conclude the proof. Now, let's proceed to verify \eqref{DCT22f}.
\begin{align}D_1(n,m)\overset{\eqref{DCT22c}}{:=}& \sum_{h=1}^m C(m-1,m-h)\sum_{2\le k_{m+1-h}\le\ldots \le k_m\le n-m} x_{n-m,k_m}\notag\\
\overset{j_r:=k_r-1}=&\sum_{h=1}^m C(m-1,m-h)\sum_{1\le j_{m+1-h}\le\ldots \le j_m\le n-m-1} x_{n-m,j_m+1}\notag\\
\overset{\eqref{DCT25a}}{=}&\sum_{h=1}^m C(m-1,m-h)\sum_{1\le j_{m+1-h}\le\ldots \le j_m\le j_{m+1}\le n-m-1} x_{n-m-1,j_{m+1}}\notag\\
\overset{p:=m-h}{=}&\ \sum_{p=0}^{m-1} C(m-1,p)\sum_{1\le j_{p+1}\le\ldots \le  j_{m+1}\le n-m-1} x_{n-m-1,j_{m+1}}\label{DCT22g}
\end{align}

Let's denote
\begin{align}
A_{m,r}&:=\sum_{1\le j_{r+1}\le\ldots\le  j_{m+1}\le n-m-1} x_{n-m-1,j_{m+1}}\notag\\
B_{m,r}&:=\sum_{2\le j_{r+1}\le\ldots\le  j_{m+1}\le n-m-1} x_{n-m-1,j_{m+1}},\ \text{for any } m\ge1,\ r\in\{0,1,\ldots,m\} \label{DCT22f0}
\end{align}
we can find
\begin{align}\label{DCT22f2}
A_{m,m}=&\sum_{1\le  j_{m+1}\le n-m-1} x_{n-m-1,j_{m+1}}
=x_{n-m-1,1}+B_{m,m}
\end{align}
and for any $p\in\{0,1,\ldots,m-1\}$,
\begin{align*}
A_{m,p}=&\sum_{1\le j_{p+1}\le\ldots\le  j_{m+1}\le n-m-1} x_{n-m-1,j_{m+1}}\\
=&\sum_{1\le j_{p+2}\le\ldots\le  j_{m+1}\le n-m-1} x_{n-m-1,j_{m+1}}+ \sum_{2\le j_{p+1}\le\ldots\le  j_{m+1}\le n-m-1} x_{n-m-1,j_{m+1}}\\
=&A_{m,p+1}+B_{m,p}
\end{align*}
By repeatedly using this formula, we obtain
\begin{align}\label{DCT22f1}
A_{m,p}=&A_{m,p+1}+B_{m,p}=A_{m,p+2}+B_{m,p+1}+B_{m,p} =\ldots=A_{m,m}+\sum_{h=p}^{m-1}B_{m,h}\notag\\
\overset{\eqref{DCT22f2}}=&x_{n-m-1,1}+ B_{m,m}+\sum_{h=p}^{m-1}B_{m,h} =x_{n-m-1,1}+ \sum_{h=p}^{m}B_{m,h}
\end{align}
Consequently,
\begin{align}\label{DCT22g0}
D_1(n,m)&\overset{\eqref{DCT22g}}=\sum_{p=0}^{m-1} C(m-1,p)A_{m,p} \overset{\eqref{DCT22f1}}{=} \sum_{p=0}^{m-1}C(m-1,p)\Big( x_{n-m-1,1} + \sum_{h=p}^{m}B_{m,h}\Big)\notag\\
&=\sum_{p=0}^{m-1}C(m-1,p)\big( x_{n-m-1,1} + B_{m,m}\big)+
\sum_{h=0}^{m-1}B_{m,h}\sum_{p=0}^{h}C(m-1,p)\notag\\
&\overset{\eqref{DCT22d1}}{=} \sum_{h=0}^{m}C(m,h)B_{m,h} +C_m x_{n-m-1,1}
\end{align}
On the other hand, we have
\begin{align}\label{DCT22c1}
&D_1(n,m+1)\overset{\eqref{DCT22c}}{:=} \sum_{h=1}^{m+1} C(m,m+1-h)\sum_{2\le k_{m+2-h}\le\ldots \le k_{m+1}\le n-(m+1)} x_{n-(m+1),k_{m+1}}\notag\\
&\overset{p:=m+1-h}{=}\sum_{p=0}^{m}C(m,p)\sum_{2\le k_ {p+1}\le\ldots\le k_{m+1}\le n-m-1}x_{n-m-1,k_{m+1}}=\sum_{p=0}^{m} C(m,p)B_{m,p}
\end{align}
Consequently, \eqref{DCT22f} is obtained as follows:
\begin{align*}
&D_1(n,m)+D_2(n,m)\overset{\eqref{DCT22c}}{=}D_1(n,m)+ \sum_{k=0}^{m-1}C_kx_{n-k-1,1}\notag\\
\overset{\eqref{DCT22g0}}{=}&
\sum_{h=0}^{m}C(m,h) B_{m,h}+C_m x_{n-m-1,1} +\sum_{k=0}^{m-1}C_kx_{n-k-1,1}\notag\\
\overset{\eqref{DCT22c1}}{=}& D_1(n,m+1)+D_2(n,m+1)
\end{align*} \end{proof}

\begin{corollary}\label{DCT25f}
If $\{x_{n,m}\}_{n\in\mathbb{N}^*,\,1\le m\le n}$ is a solution of the system \eqref{DCT25a}, then, for any $n\ge3$, the equality in \eqref{DCT25b} holds for $m=3$, specifically,
\begin{align}\label{DCT25e}
x_{n,3}=\sum_{h=0}^{n-3}C_{2}(h,h+1)x_{n-h-2,1}=\sum_{h=0}^{n-m}C_{m-1}(h,h+m-2)x_{n-m-h+1,1}\Big\vert_{m=3}
\end{align}
\end{corollary}
\begin{proof}
The second equality in \eqref{DCT25e} is straightforward to verify, and we see the first. For any $n\ge3$,
\begin{align}\label{DCT25e1}
x_{n,r}\overset{\eqref{DCT25a}}=&\sum_{k=r-1}^{n-1}x_{n-1,k}=x_{n-1,r-1}+\sum_{k=r}^{n-1}x_{n-1,k}\notag\\ \overset{\eqref {DCT25a}}=&x_{n-1,r-1}+x_{n,r+1}, \quad \text{for any } 2\le r\le n
\end{align}
Therefore, by taking $r=2$, we have
\begin{align*}
&x_{n,3}=x_{n,2}-x_{n-1,1}\overset{\eqref{DCT22b}}
=\sum_{k=0}^{n-2} C_{k}x_{n-k-1,1}-x_{n-1,1}
\overset{C_0=1}=\sum_{k=1}^{n-2} C_{k} x_{n-k-1,1}\\ \overset{h:=k-1}=&\sum_{h=0}^{n-3} C_{h+1}x_{n-h-2,1} \overset{\eqref{DCT23g6}}=\sum_{h=0}^{n-3} C_{2}(h,h+1)x_{n-h-2,1}
\end{align*}  \end{proof}

\begin{proposition}\label{DCT23} The system \eqref{DCT25a} has a solution in the form of \eqref{DCT25b}.
\end{proposition}
\begin{proof} The proof is carried out using an induction argument on both $n$ and $m$.

As a straightforward consequence of Propositions \ref{DCT25c}, Propositions \ref{DCT22} and Corollary \ref{DCT25f}, we assert that, the expression $x_{n,m}$ in the form \eqref{DCT25b} constitutes solution to the system \eqref{DCT25a} for any $(n,m)\in\{(2,2), (3,3), (3,2)\}$, i.e., for any $n\in\{2,3\}$ and $2\le m\le n$.

Suppose that our statement is verified for $n=N\ge 3$ and all $m\in\{2,\ldots, N\}$, in other words, through solving the system \eqref{DCT25a}, we have obtained:
\begin{align}\label{DCT23g2}
x_{N,m}=\sum_{h=0}^{N-m} C_{m-1}(h,h+m-2)x_{N+1-h-m,1}, \quad \text{for any } m\in\{2,\ldots, N\}
\end{align}
Now, we aim to demonstrate the assertion for $n=N+1$, i.e., we want to show, through solving the system \eqref{DCT25a}, that
\begin{align}\label{DCT23g3}
x_{N+1,m}=\sum_{h=0}^{N+1-m}C_{m-1}(h,h+m-2)x_{N+2-h-m, 1}, \quad \text{for any } m\in\{2,\ldots, N+1\}
\end{align}

It follows from Proposition \ref{DCT22} and Corollary \ref{DCT25f} that the equality in \eqref{DCT23g3} holds when $m\in\{2,3\}$.

Assuming the validity of the equality in \eqref{DCT23g3} for $m\le r\in\{3,\ldots,N\}$, we can now prove its validity for $m=r+1$ as follows:
\begin{align*}
&x_{N+1,r+1}\overset{\eqref{DCT25e1}}=x_{N+1,r}-x_{N,r-1}\notag\\
=&\sum_{k=0}^{N+1-r}C_{r-1}(k,k+r-2)x_{N+2-k-r,1}\\
&-\sum_{h=0}^{N-(r-1)}C_{r-2}(h,h+r-3)
x_{N+1-h-(r-1),1}\\
=&\sum_{h=0}^{N+1-r} \Big(C_{r-1}(h,h+r-2)- C_{r-2}(h,h+r-3)\Big) x_{N+2-h-r,1}\notag\\
\overset{\eqref{DCT23d0}}=&\sum_{h=1}^{N+1-r} \Big(C_{r-1} (h,h+r-2)- C_{r-2}(h,h+r-3)\Big) x_{N+2-h-r,1}\notag\\
\overset{k:=h-1}=&\sum_{k=0}^{N+1-(r+1)} \Big(C_{r-1} (k+1,k+r-1)- C_{r-2}(k+1,k+r-2)\Big) x_{N+2-k-(r+1),1}
\end{align*}
which, by virtue of the formula \eqref{DCT23g9}, is nothing else than the right hand side of \eqref{DCT23g3} for $m=r+1$. \end{proof} 

As a consequence of Theorem \ref{DCT25}, it can be established that,, for any $n\ge2$ and $m\in \{2,\ldots, n\}$, the following holds:
\begin{align}\label{DCT27}
&\sum_{p=0}^{n-m}C_{m}(p,p+m-1)x_{n-m-p+1,1}
\overset{\eqref{DCT25b}}=x_{n+1,m+1}\overset{\eqref{DCT25a}}=\sum_{k=m}^{n}x_{n,k}\notag\\
\overset{\eqref{DCT25b}}=& \sum_{k=m}^{n}\sum_{h=0} ^{n-k}C_{k-1}(h,h+k-2)x_{n-k-h+1,1} \notag\\
\overset{j:=k-m}=& \sum_{j=0}^{n-m}\sum_{h=0}^{n-m-j}
C_{j+m-1}(h,h+j+m-2)x_{n-m-j-h+1,1}\notag\\
\overset{p:=h+j}=& \sum_{j=0}^{n-m}\sum_{p=j}^{n-m}
C_{j+m-1}(p-j,p+m-2)x_{n-m-p+1,1}\notag\\
=&\sum_{p=0}^{n-m}x_{n-m-p+1,1}  \sum_{j=0}^{p}
C_{j+m-1}(p-j,p+m-2)
\end{align}
Therefore, we have obtained the following interesting combinatorial result:
\begin{align}\label{DCT27a}\sum_{j=0}^{k} C_{j+m-1}(k-j,k+m-2) =C_{m}(k,k+m-1),\
\text{for any } m\ge2 \text{ and } k\in \mathbb{N}
\end{align}

\section{ $(q,2)-$Fock space, its involved pair partitions and connection to the system \eqref{DCT25a}} \label{DCTsec(q,2)}

It is well--known that the sets $PP(2n)$'s and $NCPP(2n)$'s, or more generally, $PP_k(2n)$'s and $NCPP_k(2n)$'s, play a crucial role in various fields. These fields include the study of moments of certain probability distributions, exemplified by Isserlis' Theorem, as well as the normally ordered form of products involving creation--annihilation operators defined on a Fock space over a given (pre--)Hilbert space, such as Wick's Theorem, among other.

In this section, we introduce a specific Fock space, known as the $(q,2)-$Fock space, which is a special case of the $(q,m)-$Fock space introduced in \cite{YGLu2022a}. Within the context of the $(q,2)-$Fock space, we utilize a particular set of pair partitions, denoted as $\{\mathcal{P}_{n,k}\}_{n\in \mathbb{N}^*,\,1\le k\le n}$. This set of pair partitions plays a role similar to $\{NCPP_k(2n)\} _{n\in \mathbb{N}^*,\,1\le k\le n}$ for the full Fock space (for reference, see, e.g., \cite{BratteliRobinson2}, \cite{NicaSpe2006}, \cite{KRPbook1992}, \cite{ReedSimon}
and related literature) and $\{PP_k(2n)\}_{n\in \mathbb{N}^*,\,1\le k\le n}$ for the symmetrical Fock space, or more general, the $q$--Fock space with $q\in[-1,1]\setminus \{0\}$ (see, e.g., \cite{BoKumSpe97}, \cite{Bo-Spe91}). Moreover, we have the following relationship:
\begin{equation}\label{DCT05f}
NCPP(2n)\subsetneqq{\mathcal P}_{n}:=\sum_{k=1}^n
{\mathcal P}_{n,k}\subsetneqq PP(2n),\quad \text{for any } n\ge3
\end{equation}
It's worth noting that the facts $NCPP(2)=PP(2)$, $\left\vert NCPP(4) \right \vert =2$, and $\left\vert PP(4)\right\vert$ $=3$ demonstrate the non--existence of such sets ${\mathcal P}_{n}$ satisfying $NCPP(2n) \subsetneqq{\mathcal P}_{n}$ $\subsetneqq PP(2n)$ for $n\in\{1,2\}$.

The specific construction of ${\mathcal P}_{n,k}$'s and ${\mathcal P}_{n}$'s will be provided elsewhere. In this section, our primary goal is to give, by utilizing ${\mathcal P}_{n,k}$'s and ${\mathcal P}_{n}$'s, a concrete example of $\{x_{n,m}\}_{n\in\mathbb{N}^*, \,1 \le m\le n}$ that satisfies the system \eqref{DCT25a} with a boundary condition distinct from \eqref{DCT-CTSbou1}.         

\subsection{The usual Fock space and pair partitions}
\label{DCTsec(q,m)00}
Recalling that, let $\mathcal{H}$ be a (pre--)Hilbert space over $\mathbb{C}$ with the scalar product $\langle \cdot,\cdot\rangle$, one defines the full (or free) Fock space over $\mathcal{H}$ as $\Gamma_{free} (\mathcal{H}):= \bigoplus_{n=0}^\infty \mathcal{H} ^{\otimes n}$. Hereinafter, $\mathcal{H}^{\otimes 0}:=\mathbb{C}$ and $\mathcal{H}^{\otimes n}:=$the $n-$fold tensor product of $\mathcal{H}$ for any $n\in\mathbb{N}^*$.

On the full Fock space $\Gamma_{free}(\mathcal{H})$, for any $f\in \mathcal{H}$, one defines the {\it creation operator with the test function $f$}, denoted as $b^+_0(f)$, by the linearity and
\begin{align}\label{DCT28}
b^+_0(f)\Phi&:=f\notag\\
b^+_0(f)(f_1\otimes \ldots\otimes f_n)&:=f\otimes f_1\otimes \ldots\otimes f_n,\ \text{for any } n\ge 1\text{ and } f_1,\ldots,f_n\in\mathcal{H}
\end{align}

The creation operator $b^+_0(f)$ is bounded (in fact $b^+_0(f)=\Vert f\Vert$), and thus, its conjugate $b_0(f):=\big(b^+_0(f)\big)^*$ is well--defined. The action of and $b_0(f)$ is surely given by:
\begin{align}\label{DCT28a}
b_0(f)\Phi&=0\notag\\
b_0(f)(f_1\otimes \ldots\otimes f_n)&= \langle f, f_1\rangle f_2\otimes \ldots\otimes f_n,\ \text{for any } n\ge 1\text{ and } f_1,\ldots,f_n\in\mathcal{H}
\end{align}

By replacing the tensor product ``$\otimes$'' with the {\it symmetrical tensor product} ``$\circ$'', one defines the symmetrical (or Bosonic) Fock space over $\mathcal{H}$ as $\Gamma_{symm}(\mathcal{H}):= \bigoplus_{n=0}^\infty \mathcal{H}^{\circ n}$. On the symmetrical Fock space $\Gamma_{symm}(\mathcal{H})$, the creation operator with the test function $f\in \mathcal{H}$, denoted as $b_1^+(f)$, is defined in similar way as that $b_0^+(f)$ by linearity and the following action:
\begin{align}\label{DCT28b}
b^+_1(f)\Phi&:=f\notag\\
b^+_1(f)(f_1\circ \ldots\circ f_n)&:=f\circ f_1\circ \ldots\circ f_n,\quad\text{for any } n\ge 1\text{ and } f_1,\ldots,f_n\in\mathcal{H}
\end{align}
It is important to note that this operator is not bounded on $\Gamma_{symm} (\mathcal{H})$ (it is a bounded operator from $\mathcal{H}^{\circ n}$ to $\mathcal{H}^{\circ (n+1)}$ for each $n$), while it has a well--defined adjoint, denoted as $b_1(f):= \big(b^+_1(f)\big)^*$. This adjoint is called the annihilation operator with the test function $f$.

In what follows, one denotes, for any $n\in\mathbb{N}^*$,
\begin{align}\label{DCT02c00}
&\{-1,1\}^{n}:=\big\{\text{functions on }\{1,\ldots,n\} \text{ and valued in }\{-1,1\} \big\}\notag\\
&\{-1,1\}^{2n}_+:=\big\{\varepsilon\in \{-1,1\}^{2n} :\sum_{h=1}^{2n}\varepsilon(h) =0,\, \sum_{h=p}^{2n}\varepsilon(h)\ge 0,\ \text{for any } p\in \{1,\ldots,2n\}\big\}  \notag\\
&\{-1,1\}^{2n}_-:=\{-1,1\}^{2n}\setminus \{-1,1\}^{2n}_+
\end{align}

\begin{remark}\label{DCT-rem4-0} Clearly, in the definition of $\{-1,1\}^{2n}_+$, the condition $\sum_{h=p}^{2n}\varepsilon(h)\ge 0$ for any $p\in \{1,\ldots,2n\}$ can be replaced by $\sum_{h=1}^{k} \varepsilon(h)\le 0$ for any $k\in \{1,\ldots,2n\}$ since $\sum_{h=1}^{2n} \varepsilon(h)=0$.
\end{remark}

For any $n\in\mathbb{N}^*$, one introduces a function $\tau: PP(2n)\longmapsto \{-1,1\}^{2n}$ such that, for any $\theta:=\{(l_h,r_h)\}_{h=1}^n\in PP(2n)$, $\varepsilon:=\tau(\theta)$ is an element of $PP(2n)$ defined as follows:
\begin{align}\label{DCT02c03}
\varepsilon(j):=\tau(\theta)(j):=\begin{cases}
1,&\text{ if }j\in \{r_1,r_2,\ldots, r_n\}\\
-1,&\text{ if }j\in \{l_1,l_2,\ldots, l_n\}\\
\end{cases},\quad \text{for any } j\in\{1,\ldots,n\}
\end{align}
It is obvious that $\tau$ maps $PP(2n)$ into $\{-1,1\}^{2n}_+$ due to the following facts:

$\bullet$ $\sum_{h=1}^{2n}\varepsilon(h)=\big\vert\big\{r_h:
h\in\{1, \ldots,n\}\big\}\big\vert-\big\vert
\big\{l_h:h\in\{1, \ldots,n\}\big\}\big\vert=n-n=0$;

$\bullet$ $\sum_{h=p}^{2n}\varepsilon(h)\ge 0$ for any $p\in \{1,2,\ldots,2n\}$, since $r_j>l_j$ for any $j\in\{1,\ldots,n\}$.

Moreover, for any $n\in\mathbb{N}^*$ and $\varepsilon\in \{-1,1\}^{2n}_+$, as demonstrated in \cite{Ac-Lu96} and \cite{Ac-Lu2022a},

$\bullet$ the set $\tau^{-1}(\varepsilon)\cap NCPP(2n)$ has a cardinality 1, i.e. $\tau$ induces a bijection between $NCPP(2n)$ and $\{-1,1\}^{2n}_+$;

$\bullet$ the set
\begin{align}\label{DCT02c04}
&PP(2n,\varepsilon):=\tau^{-1}(\varepsilon)\notag\\
:=&\big\{\{(l_h,r_h)\}_{h=1}^n\in PP(2n): \varepsilon(l_h)=-1,\, \varepsilon(r_h)=1, \text{ for any } h=1,\ldots,n\big\}
\end{align}
has a cardinality of $\prod_{h=1}^n (2h-l_h)$.

In the following, for any $n\in\mathbb{N}^*$ and for any $\theta:=\{(l_h,r_h)\}_{h=1}^n\in NCPP(2n)$, one refers to $\tau(\theta)\in \{-1,1\}^{2n}_+$ as the {\bf counterpart} of $\theta$; the unique element of  $\tau^{-1}(\varepsilon)\cap NCPP(2n)$ is called the {\bf counterpart} of $\varepsilon$. Which will be denoted as $\{(l^\varepsilon_h, r^\varepsilon_h)\}_{h=1}^n$ unless otherwise specified.

\begin{remark}\label{DCT-rem4-1} For any
$n\in\mathbb{N}^*$ and $\varepsilon\in\{-1,1\}^{2n}_+$, it is important to note that all elements in $PP(2n,\varepsilon)$ must have the same left indices. Specifically, they consists of $n$ elements of the set $\varepsilon^{-1}(\{-1\})$ in increasing order.
\end{remark}

By denoting
\begin{align}\label{DCT02c07}
b_q^\varepsilon(f):=\begin{cases}b_q^+(f),&\text{ if } \varepsilon=1\\ b_q(f),&\text{ if } \varepsilon=-1
\end{cases},\quad \text{for any } q\in\{0,1\}\text{ and }f\in\mathcal{H}
\end{align}
it is well--known that, with respect to the vacuum state, the expectation of an arbitrary product of the creation--annihilation operators (i.e., the {\it vacuum mixed--moments}) is determined by $PP(2n)$'s and $NCPP(2n)$'s. By denoting  $\Phi:=1\oplus0\oplus0\oplus\ldots$ as the {\it vacuum vector} for both $\Gamma_{free}(\mathcal{H})$ and $\Gamma_{symm}(\mathcal{H})$, for any $m\in\mathbb{N} ^*$ and $\{f_1,\ldots, f_m\}\subset\mathcal{H} $,
\begin{align}\label{DCT02c02}
&\big\langle \Phi,b^{\varepsilon(1)}_q(f_1)\ldots b^{\varepsilon(m)}_q(f_m)\Phi\big\rangle\notag\\
=&\begin{cases} \sum_{\{(l_h,r_h)\}_{h=1}^n\in PP(2n,\varepsilon)}
\prod_{h=1}^n\langle f_{l_h},f_{r_h}\rangle,&\text{ if } m=2n, \varepsilon\in  \{-1,1\}^{2n}_+\text{ and }q=1\\ \prod_{h=1}^n\langle f_{l^\varepsilon_h}, f_{r^\varepsilon_h}\rangle,&     \text{ if } m=2n, \varepsilon\in  \{-1,1\}^{2n}_+ \text{ and }q=0\\
 0,&\text{ otherwise}\\
\end{cases}
\end{align}
Where, it is clear that the case ``otherwise'' means either $m$ is odd or $m=2n,\,\varepsilon\in  \{-1,1\}^{2n}_-$. In particular, for any $f\in\mathcal{H} $,
\begin{align}\label{DCT02c01}
\big\langle \Phi,\big(b_q(f)+b^+_q(f)\big)^m\Phi \big\rangle  =&\sum_{\varepsilon\in\{-1,1\}^m}\big\langle \Phi,b^{\varepsilon(1)}_q(f)\ldots b^{\varepsilon(m)}_q(f) \Phi\big\rangle   \notag\\
=&\Vert f\Vert^{2n}\cdot\begin{cases} \big\vert PP(2n)\big\vert,&\text{ if } m=2n \text{ and }q=1\\ \big\vert NCPP(2n)\big\vert,& \text{ if }m=2n \text{ and }q=0\\
0,&\text{ if } m\text{ is odd}
\end{cases}
\end{align}

It is interesting to notice that, for any $k\in\{1,\ldots,n\}$, the application $\tau$ given in \eqref{DCT02c03} induces a bijection between $NCPP_k(2n)$ and the set
\begin{align}\label{DCT02c05}
&\{-1,1\}^{2n,k}_+:=\big\{\varepsilon\in  \{-1,1\}^{2n}_+ :\, \max\varepsilon^{-1}(\{-1\})=2n-k\big\}\notag\\
=&\big\{\varepsilon\in\{-1,1\}^{2n}_+:\,\varepsilon (2n-k) =-1,\,\varepsilon(j)=1\ \ \text{for any } j>2n-k \big\}
\end{align}
Therefore,
\begin{align}\label{DCT02c06}
&\sum_{\varepsilon\in\{-1,1\}^{2n,k}}\big\langle \Phi, b^{\varepsilon(1)}_0(f)\ldots b^{\varepsilon(2n)}_0(f) \Phi\big\rangle  \notag\\
=&\sum_{\varepsilon\in\{-1,1\}^{2n-k-1}}\big\langle \Phi,b^{\varepsilon(1)}_0(f)\ldots b^{\varepsilon (2n- k-1)}_0(f)b_0(f)\big(b^+_0(f)\big)^k\Phi\big\rangle
\notag\\
=&\Vert f\Vert^{2n}\cdot\big\vert NCPP_k(2n)\big\vert
\overset{\eqref{DCT05a}}=\Vert f\Vert^{2n}\cdot C(n-1,n-k)
\end{align}

\subsection{The $(q,2)-$Fock space and some of its  elementary properties}\label{DCTsec(q,m)01}\medskip

Let $\mathcal{H}$ be a Hilbert space with the scalar product $\langle \cdot,\cdot \rangle$  of dimension greater than or equal to 2 (we maintain this convention throughout unless otherwise stated), let $\mathcal{H}^{\otimes n}$ be its $n-$fold tensor product. For any $q\in[-1,1]$, one introduces the following:

$\bullet$ $\lambda_1:={\bf 1}_{\mathcal{H}}$, i.e., the identity operator on $\mathcal{H}$;

$\bullet$ $\lambda_2$ is the linear operator on $\mathcal{H}^{\otimes (n+2)}$ defined as
\begin{equation}\label{DCT05f0}
\lambda_2(f\otimes g):=f\otimes g+q g\otimes f\,,\quad \text{for any } f,g\in\mathcal{H}
\end{equation}

$\bullet$ For any $n\in\mathbb{N}^*$, $\lambda_{n+2}$ is the linear operator on $\mathcal{H}^{\otimes (n+2)}$defined as
\begin{equation}\label{DCT05f1}
\lambda_{n+2}:={\bf 1}_{\mathcal{H}}^{\otimes n}\otimes \lambda_2
\end{equation}
The easily verified positivity of $\lambda_n$'s implies that the completion of the quotient space \[\big(\mathcal{H}^{\otimes n}, \langle \cdot,\lambda_{n}\cdot \rangle_{\otimes n}\big)/Ker\langle \cdot,\lambda_{n}\cdot \rangle_{\otimes n}\]
denoted as $\mathcal{H}_n$, is a Hilbert space, where $\langle \cdot,\cdot \rangle_{\otimes n}$ is the usual tensor scalar product. By denoting $\langle \cdot,\cdot \rangle_{n}$ as the scalar product of $\mathcal{H}_n$, one has $\langle \cdot,\cdot \rangle_{1}:=\langle \cdot,\cdot \rangle $, and for any $n\ge2$,
\begin{equation}\label{DCT05f2}
\langle F,G \rangle_{n}:=\langle F,\lambda_nG \rangle_{\otimes n} \,,\quad\text{for any } F,G\in \mathcal{H}^{\otimes n}
\end{equation}
or equivalently for any $n\in \mathbb{N}$,
\begin{align}
\langle F,G\otimes f\otimes g \rangle_{n}:=&\langle F,G\otimes f\otimes g\rangle_{\otimes n}+ q\langle F,G\otimes g\otimes f\rangle_{\otimes n}\notag\\
&\text{for any } F\in \mathcal{H}^{\otimes (n+2)},\, G\in \mathcal{H}^{\otimes n}\text{ and }f,g\in \mathcal{H}\label{DCT05f3}
\end{align}

\begin{definition}\label{(q,2)-Fock} Let $\mathcal{H}$ be a Hilbert space and let $\mathcal{H}_n$ be the Hilbert space defined as above for any $n\in\mathbb{N}^*$. We denote $\mathcal{H}_{0} := \mathbb{C}$ and define the Hilbert space $\Gamma_{q,2} (\mathcal{H}):=\bigoplus_{n=0}\mathcal{H}_n $. This Hilbert space is referred to as the {\bf (q,2)--Fock space} over $\mathcal{H}$; the vector $\Phi:=1\oplus0 \oplus0 \oplus\ldots$ is termed as the {\bf vacuum vector} of $\Gamma_{q,2}(\mathcal {H})$; and for any $n\in\mathbb{N}^*$, $\mathcal{H}_n$ is called the {\bf $n-$particle space}.
\end{definition}

\begin{remark}\label{consist}1) If there is no confusion, we will simply use $\langle \cdot,\cdot \rangle$ and $\Vert\cdot\Vert$ to denote the scalar product and the induced norm on both $\Gamma_{q,2}(\mathcal{H})$ and $\mathcal{H}_n$'s.

2) It is straightforward to see the {\bf consistency} of $\langle \cdot,\cdot \rangle_{n}$'s: for any $0\ne f\in\mathcal{H}$ and  $n\in\mathbb{N}^*$,
\begin{equation}\label{DCT05f4}
\Vert f\otimes F\Vert=0\text{ whenever } F\in\mathcal{H}_{n} \text{ verifying } \Vert F\Vert=0
\end{equation}
\end{remark}

\begin{definition}\label{creaOn(q,2)} For any $f\in \mathcal{H}$, the {\bf creation operator} (with the test function $f$), denoted as $A^+(f)$, is defined as a {\bf linear} operator on $\Gamma_{q,2}(\mathcal{H})$ such that
\begin{equation}\label{creaOn(q,2)a}
A^+(f)\Phi:=f\,,\quad A^+(f)F:=f\otimes F,\quad \text{for any } n\in\mathbb{N}^*\text{ and }F\in\mathcal{H}_n
\end{equation}
\end{definition}

\begin{remark}\label{DCTrem3-2}The consistency  mentioned in Remark \ref{consist} ensures that the operator $A^+(f)$ defined above is a well--defined {\bf linear} operator since it brings the zero of $\mathcal{H}_n$ to the zero of $\mathcal{H}_{n+1}$.
\end{remark}

\begin{proposition}\label{DCT05g}Let $\mathcal{H}$ be a Hilbert space and $q\in[-1,1]$. For any $f\in\mathcal{H}$, the $(q,2)-$creation operator $A^+(f)$ is bounded with the norm given by:
\begin{align}
\Vert A^+(f)\Vert =\Vert f\Vert\cdot\begin{cases} \sqrt{1+q},&\text{ if }q\in[0,1];\\ 1,&\text{ if }q\in[-1,0)  \end{cases}\label{DCT05g0}
\end{align}
Moreover, the {\bf $(q,2)-$annihilation operator} with the test function $f\in \mathcal{H}$, defined as $A(f):=\big(A^+(f)\big)^*$, satisfies the following:

1) $A(f)\Phi=0$ and for any $n\in\mathbb{N}^*$, $\{g_1,\ldots, g_n\}\subset\mathcal{H}$,
\begin{align}\label{DCT05g1}
&A(f)(g_1\otimes\ldots\otimes g_n)
=A(f)A^+(g_1)\ldots A^+(g_1)\Phi\notag\\
=&\begin{cases} \langle f, g_1\rangle\Phi, &\text{ if }n=1;\\
\langle f, g_1\rangle g_2+q\langle f, g_2\rangle g_1,&\text{ if }n=2;\\     \langle f, g_1\rangle g_2\otimes\ldots\otimes g_n,&\text{ if }n> 2  \end{cases}
\end{align}

2) $\Vert A(f)\big\vert_{\mathcal{H}_{1}}\Vert =\Vert A^+(f)\big\vert_{\mathcal{H}_{0}}\Vert =\Vert f\Vert$ and for any $n\in\mathbb{N}^*$,
\begin{align}\label{DCT05g2}
&\Vert A(f)\big\vert_{\mathcal{H}_{n+1}}\Vert =\Vert A^+(f)\big\vert_{\mathcal{H}_n}\Vert\notag\\ &
\Vert A(f)A^+(f)\big\vert_{\mathcal{H}_n}\Vert  =\Vert A^+(f)A(f)\big\vert_{\mathcal{H}_n}\Vert =\Vert A(f)\big\vert_{\mathcal{H}_n}\Vert ^2
\end{align}
Consequently
\begin{align}\label{DCT05g3}
\Vert A(f)\Vert =\Vert A^+(f)\Vert\,;\quad \Vert A(f)A^+(f)\Vert  =\Vert A^+(f)A(f)\Vert =\Vert A(f)\Vert^2
\end{align}
\end{proposition}
\begin{proof}The boundedness of the $(q,2)-$creation operator $A^+(f)$ is straightforward since it can be expressed as:
\begin{equation}\label{DCT05g4}
A^+(f)=b^+_q(f)\big\vert_{\mathcal{H}_0\oplus\mathcal{H}_1}\oplus b^+_0(f)\big\vert_{\mathcal{H}_{2}\oplus \mathcal{H}_{3}\oplus\ldots}
\end{equation}
Hereinafter, $b^+_q(f)$ is the creation operator with the test function $f$ on the $q-$Fock space for any $q\in[-1,1]$. Consequently, its adjoint exists and satisfies the equalities mentioned in the affirmation 1), as demonstrated in, e.g., \cite{Bo-Spe91}. Affirmation 2) is trivial since each $\mathcal{H}_n$ is a Hilbert space. Now let's prove \eqref{DCT05g0}.

It is well--known (as seen in \cite{Bo-Spe91} for instance) that
\begin{align}\label{DCT05g5}
\Vert A^+(f)\big\vert_{\mathcal{H}_n}\Vert =\Vert f\Vert\cdot\begin{cases} \sqrt{1+q},&\text{ if }n=1;\\ 1,&\text{ if }n\ne 1  \end{cases}
\end{align}
and this leads to
\begin{align*}
\Vert A^+(f)\Vert =\Vert f\Vert\cdot\max\big\{1, \sqrt{1+q}\big\}= \Vert f\Vert\cdot \begin{cases} \sqrt{1+q},&\text{ if }q\in[0,1];\\ 1,&\text{ if }q\in[-1,0)  \end{cases}
\end{align*}
Thus the proof is successfully completed. \end{proof}

We will use the following analogue of \eqref{DCT02c07}:
\begin{align*}
A^\varepsilon(f):=\begin{cases}A^+(f),&\text{ if } \varepsilon=1\\ A(f),&\text{ if } \varepsilon=-1
\end{cases}\,,\qquad  \text{ for any }f\in\mathcal{H}
\end{align*}
In analogy of the formula \eqref{DCT02c02}, let's examine the vacuum expectations (mixed moments)
\begin{align}\label{DCT29a}
\big\langle \Phi,A^{\varepsilon(1)}(f_1)\ldots A^{\varepsilon(m)}(f_m)\Phi\big\rangle
\end{align}
for $m\in\mathbb{N}^*$, $\varepsilon\in  \{-1,1\}^{m}$ and $\{f_1,\ldots, f_m\}\subset\mathcal{H} $.

\begin{proposition}\label{DCT29z}On the $(q,2)-$Fock space, the vacuum expectation \eqref{DCT29a} equals to zero in either the following cases:

$\bullet$ $m$ is odd;

$\bullet$ $m=2n$ with $n\in\mathbb{N}^*$ and  $\varepsilon\in  \{-1,1\}^{2n}_-$ (see \eqref{DCT02c00} for the definition).
\end{proposition}
\begin{proof} Thanks to the following facts
\begin{align}\label{DCT29c}
\mathcal{H}_n\overset{A^+(f)}{\longmapsto} \mathcal{H}_{n+1}\overset{A(g)}{\longmapsto}\mathcal{H}_n, \quad A(g)\Phi=0,\quad\text{for any } n\in\mathbb{N},\ f,g\in\mathcal{H}
\end{align}
we know that for any $\varepsilon\in  \{-1,1\}^{m}$, the vector $A^{\varepsilon(1)}(f_1)\ldots A^{\varepsilon(m)}(f_m) \Phi$ differs from zero only if
\begin{align}\label{DCT29c1}
\sum_{k=p}^m \varepsilon(k)\ge0,\quad\text{for any } p\in \{1,\ldots, m\}
\end{align}
This condition is required because otherwise, an annihilation operator acts, sooner or later, on the vacuum vector. Similarly, the vector \[\big(A^{\varepsilon(1)}(f_1)\ldots A^{\varepsilon(m)}(f_m)\big)^*\Phi=A^{-\varepsilon(m)}(f_m)\ldots A^{-\varepsilon(1)}(f_1)\Phi\]
differs from zero only if
\begin{align}\label{DCT29c2}
\sum_{k=1}^p \varepsilon(k)\le0,\quad\text{for any } p\in \{1,\ldots, m\}
\end{align}
In particular, the vacuum mixed moment \eqref{DCT29a} differs from zero only if, by taking $p=1$ in \eqref{DCT29c1} and $p=m$ in \eqref{DCT29c2},
\begin{align}\label{DCT29c3}
\sum_{k=1}^m \varepsilon(k)=0
\end{align}
this clearly necessitates that $m$ must be even.

In the case of $m=2n$, \eqref{DCT29c3} becomes to respectively $\sum_{k=1}^{2n} \varepsilon(k)=0$ and \eqref{DCT29c1} become to $\sum_{k=p}^{2n} \varepsilon(k)\ge0$ for all $p\in \{1,\ldots, 2n\}$. Thus the vacuum mixed moment \eqref{DCT29a} differs from zero only if $\varepsilon\in  \{-1,1\}^{2n}_+$. \end{proof}

Thanks to Proposition \ref{DCT29z}, the analysis of mixed moments is simplified to consider
\begin{align}\label{DCT29b1}
\big\langle \Phi,A^{\varepsilon(1)}(f_1)\ldots A^{\varepsilon(2n)}(f_{2n})\Phi\big\rangle
\end{align}
The computation of this expression practically yields the $\mathcal{P}_n$'s mentioned in the beginning of the present section (with its explicit formulation given in \cite{YGLu2022c}). To illustrate this point effectively,
we will evaluate the expression \eqref{DCT29b1} for $n\in\{1,2,3\}$ in the case of $q=1$. Additionally, we will use the notation $\varepsilon=(\varepsilon(1), \ldots,\varepsilon(m))$ for any $m\in\{\mathbb N\}^*$ and $\varepsilon\in \{-1,1\}^{m}$.

For $n=1$, the set $\{-1,1\}^{2}_+$ consists of precisely one element $\varepsilon =(-1,1)$, and its counterpart is $\{(1,2)\}\in PP(2,\varepsilon)$. So, by denoting $(l_1,r_1)=(1,2)$, we have
\begin{align}\label{DCT29d0}
\big\langle \Phi,A^{\varepsilon(1)}(f_1) A^{\varepsilon(2)}(f_{2})\Phi\big\rangle=&\big\langle \Phi,A(f_1) A^+(f_{2})\Phi\big\rangle\notag\\ \overset{\eqref{DCT05g1}}=&\langle f_1, f_2\rangle=
\sum_{\{(l_1,r_1)\}\in PP(2,\varepsilon)}
\langle f_{l_1}, f_{r_1}\rangle
\end{align}

For $n=2$, the set $\{-1,1\}^{4}_+$ consists of two elements: $\varepsilon_1=(-1,1,-1,1)$ and $\varepsilon_2= (-1,-1,1,1)$. Specifically, $PP(4, \varepsilon _1)$ consists of exactly one element $\{(1,2),(3,4)\}$. So, by denoting $(l_1,r_1)=(1,2)$ and $(l_2,r_2)=(3,4)$, we find
\begin{align}\label{DCT29d1}
&\big\langle \Phi,A^{\varepsilon_1(1)}(f_1) A^{\varepsilon_1(2)}(f_{2})A^{\varepsilon_1(3)}(f_3) A^{\varepsilon_1(4)}(f_{4})\Phi\big\rangle\notag\\
\overset{\eqref{DCT05g1}}=&\big\langle \Phi,A(f_1) A^+(f_{2})A(f_3) A^+(f_{4})\Phi\big\rangle \notag\\ 
=&\langle f_1, f_2\rangle \cdot \langle f_3, f_4\rangle=\sum_{\{(l_h,r_h)\}_{h=1}^2\in PP(4, \varepsilon _1)} \prod_{h=1}^2\langle f_{l_h}, f_{r_h}\rangle
\end{align}
Similarly, due to the fact that $PP(4,\varepsilon_2)$ consists of two elements $\{(1,4), (2,3)\}$ and $\{(1,3), (2,4)\}$, we obtain
\begin{align}\label{DCT29d2}
&\big\langle \Phi,A^{\varepsilon_2(1)}(f_1) A^{\varepsilon_2(2)}(f_{2})A^{\varepsilon_2(3)}(f_3) A^{\varepsilon_2(4)}(f_{4})\Phi\big\rangle\notag\\
=&\big\langle \Phi,A(f_1) A(f_{2})A^+(f_3) A^+(f_{4}) \Phi \big\rangle \overset{\eqref{DCT05g1}}=\langle f_1,f_4\rangle \cdot \langle f_2, f_3\rangle+ \langle f_1, f_3\rangle \cdot \langle f_2, f_4\rangle\notag\\
=&\sum_{\{(l_h,r_h)\}_{h=1}^2\in PP(4, \varepsilon _2)} \prod_{h=1}^2\langle f_{l_h}, f_{r_h}\rangle
\end{align}

Now, let's examine the expression \eqref{DCT29b1} for $n=3$ with $\varepsilon=(-1,-1,-1,1,1,1)$.
\begin{align}\label{DCT29d3}
&\big\langle \Phi,A^{\varepsilon(1)}(f_1) A^{\varepsilon(2)}(f_{2})A^{\varepsilon(3)}(f_3) A^{\varepsilon(4)}(f_{4}) A^{\varepsilon(5)}(f_{5}) A^{\varepsilon(6)}(f_{6})\Phi\big\rangle\notag\\
=&\big\langle \Phi,A(f_1) A(f_{2})A(f_3) A^+(f_{4})A^+(f_5) A^+(f_{6})\Phi\big\rangle\notag\\
\overset{\eqref{DCT05g1}}=&
\langle f_3, f_4\rangle \cdot\big\langle \Phi,A(f_1) A(f_{2})A^+(f_5) A^+(f_{6})\Phi\big\rangle\notag\\
\overset{\eqref{DCT05g1}}=& \langle f_1, f_6\rangle \cdot\langle f_2, f_5\rangle\cdot \langle f_3, f_4\rangle+\langle f_1, f_5\rangle\cdot \langle f_2, f_6\rangle\cdot\langle f_3, f_4\rangle
\end{align}
Let $\mathcal{P}_{3,\varepsilon}:=\big\{\{(1,6),(2,5), (3,4)\}, \{(1,5), (2,6),(3,4)\}\big\}$, the expression in \eqref{DCT29d3} is equal to
\begin{align*}
\sum_{\{(l_h,r_h)\}_{h=1}^3\in{\mathcal P}_{3,\varepsilon} } \prod_{h=1}^3\langle f_{l_h},f_{r_h}\rangle
\end{align*}
Furthermore, 
\begin{align*}&\{(1,6),(2,5),(3,4)\}\\
=&NCPP(6,\varepsilon)\subsetneqq {\mathcal P}_{3,\varepsilon}\subsetneqq PP(6,\varepsilon)=\big\{ \{(h,\sigma(h))\}_{h=1}^3 :\,\sigma\in{\mathfrak S}_n \}\end{align*}

In fact, it is proved in \cite{YGLu2022c}, for any $n\in{\mathbb N}^*$ and $\varepsilon\in  \{-1,1\}^{2n}_+$, there exists unique ${\mathcal P}_{n,\varepsilon}$ (with its explicit formulation provided) such that:

$\bullet$ \eqref{DCT05f} holds;

$\bullet$ for any $\{f_1,\ldots,f_{2n}\}\subset{\mathcal H}$, the mixed moment \eqref{DCT29b1} equals to
\begin{align}\label{DCT29d5x}
\sum_{\{(l_h,r_h)\}_{h=1}^n\in{\mathcal P}_{n,\varepsilon} } \prod_{h=1}^n\langle f_{l_h},f_{r_h}\rangle
\end{align}

As a consequence, for any $f\in{\mathcal H}$,
\begin{align}\label{DCT29d5y}
&\langle\Phi,\big(A(f)+A^+(f)\big)^m\rangle\notag\\
=&\begin{cases}0,&\text{ if $m$ is odd}\\ \sum_{\varepsilon\in\{-1,1\}^{2n} _{+} } \sum_ {\{(l_h,r_h)\}_{h=1}^n\in\mathcal{P}_{n,\varepsilon} } \prod_{h=1}^n\langle f,f\rangle,&\text{ if } m=2n \end{cases}\notag\\
=&\Vert f\Vert^{2n} \begin{cases}0,&\text{ if $m$ is odd}\\ \sum_{\varepsilon\in\{-1,1\}^{2n} _{+} } \vert {\mathcal P}_{n,\varepsilon}\vert,&\text{ if } m=2n \vert \end{cases}
\end{align}
In \cite{YGLu2022b}, the vacuum distribution of the {\it field operator} $A(f)+A^+(f)$ is calculated by using this result and the explicit formulation of ${\mathcal P}_{n,\varepsilon}$'s provided in \cite{YGLu2022c}.

To delve into this further, let's introduce
\begin{align}\label{DCT29d5z1}
{\mathcal P}_{n,k}:=\bigcup_{\varepsilon\in \{-1,1\}^{2n}_{+,k}}{\mathcal P}_{n,\varepsilon}, \qquad\text{for any } n\in{\mathbb N}^*,\, k\in\{1,\ldots,n\}
\end{align}
Due to the equality $\{-1,1\}^{2n} _{+}= \cup_{k=1}^n\{-1,1\}^{2n} _{+,k}$ and the disjointness of $\{-1,1\}^{2n} _{+,k}$'s, one knows that, for any $n\in{\mathbb N}^*$, ${\mathcal P}_{n,k}$'s are disjoint and
\begin{align}\label{DCT29d5z}
&{\mathcal P}_n:=\bigcup_{\varepsilon\in \{-1,1\}^{2n}_+}
{\mathcal P}_{n,\varepsilon}=\bigcup_{k=1}^n \bigcup_{\varepsilon\in \{-1,1\}^{2n}_{+,k}}{\mathcal P}_ {n,\varepsilon}=\bigcup_{k=1}^n{\mathcal P}_{n,k}
\end{align}
In particular, by defining
\begin{align}\label{DCT29d5z2}
P_n:=\vert {\mathcal P}_n\vert\text{ and } P_{n,k}:= \vert{\mathcal P}_{n,k}\vert, \quad\text{ for any } n\in{\mathbb N}^*,\, k\in\{1,\ldots,n\}
\end{align}
one can deduce, thanks to the disjointness of ${\mathcal P}_{n,k}$'s, that
\begin{align}\label{DCT41a2}
P_n=\sum_{k=1}^nP_{n,k},\ 
P_{n,k}=\sum_ {\varepsilon \in\{-1,1\}^{2n} _{+,k} }
\vert{\mathcal P}_{n,\varepsilon}\vert , \ \text{ for any }n\in{\mathbb N}^*,\ k\in\{1,\ldots,n\}
\end{align}
Moreover, for any  $f\in{\mathcal H}^{(1)}:=$the set of all unit--norm elements of ${\mathcal H}$, one finds
\begin{align}\label{DCT41a3}
\vert{\mathcal P}_{n,\varepsilon}\vert
=\big\langle \Phi, A^{\varepsilon(1)}(f)\ldots A^{\varepsilon(2n)} (f) \Phi\big\rangle
\end{align}
Therefore, by combining \eqref{DCT41a2} and \eqref{DCT41a3}, it is concluded that, for any $n\in{\mathbb N}^*,\, k\in\{1,\ldots,n\}$ and $f\in{\mathcal H}^{(1)}$,
\begin{align}\label{DCT41a5}
P_{n,k}=\sum_ {\varepsilon \in\{-1,1\}^{2n} _{+,k} }\big\langle \Phi, A^{\varepsilon(1)}(f)\ldots A^{\varepsilon(2n)} (f) \Phi\big\rangle
\end{align}

Now let's calculate $P_{n,k}$'s.
It is straightforward to compute, for any $f\in \mathcal{H}^{(1)}$, that:
\[P_{1,1}=\sum_{\varepsilon\in\{-1,1\}^2_+}\big\langle \Phi, A^{\varepsilon(1)}(f)A^{\varepsilon(2)} (f) \Phi\big\rangle=\big\langle \Phi, A^{\varepsilon(1)}(f)A^{\varepsilon(2)} (f) \Phi\big\rangle\overset{\eqref{DCT05g1}}=1
\]
\begin{align*}
P_{2,1}=&\sum_{\varepsilon\in\{-1,1\}^2_{+,1}}\big\langle \Phi, A^{\varepsilon(1)}(f)A^{\varepsilon(2)} (f)A^{\varepsilon(3)} (f)A^{\varepsilon(4)} (f) \Phi\big\rangle\\
=&\big\langle \Phi, A(f)A^+(f)A(f)A^+(f)
\Phi\big\rangle\overset{\eqref{DCT05g1}}=1
\end{align*}
and
\begin{align}\label{DCT41a6}
P_{2,2}&=\sum_{\varepsilon\in\{-1,1\}^2_{+,2}}\big\langle \Phi, A^{\varepsilon(1)}(f)A^{\varepsilon(2)} (f)A^{\varepsilon(3)} (f)A^{\varepsilon(4)} (f) \Phi\big\rangle\notag\\
&=\big\langle \Phi, A(f)A(f)A^+(f)A^+(f) \Phi\big\rangle\overset{\eqref{DCT05g1}}=1+q
\end{align}
\begin{proposition}\label{DCT50}Using the notations and conventions established above, we have
\begin{align}\label{DCT50a}
P_{n+1,1}=P_n \text{ and } P_{n+1,2}=(1+q)P_n,\quad \text{for all } n\in{\mathbb N}^*
\end{align}
\end{proposition}
\begin{proof} For any $n\in{\mathbb N}^*$ and $f\in \mathcal{H}^{(1)}$, \eqref{DCT41a5} yields the following result:
\begin{align}\label{DCT50b}
&P_{n+1,k}=\sum_ {\varepsilon \in\{-1,1\}^{2n+2} _{+,k} }\big\langle \Phi, A^{\varepsilon(1)}(f)\ldots A^{\varepsilon(2n+2)} (f) \Phi\big\rangle\notag\\
=&\begin{cases}
\sum_ {\varepsilon \in\{-1,1\}^{2n+2} _{+,1} }
\big\langle \Phi, A^{\varepsilon(1)}(f)\ldots A^{\varepsilon(2n)} (f)A(f)A^+ (f) \Phi\big\rangle,&\text{ if }k=1\\
\sum_ {\varepsilon \in\{-1,1\}^{2n+2} _{+,2} }
\big\langle \Phi, A^{\varepsilon(1)}(f)\ldots A^{\varepsilon(2n-1)} (f)A(f)\big(A^+ (f)\big)^2 \Phi\big\rangle,&\text{ if }k=2
\end{cases}
\end{align}
For $k=1$ and for any $\varepsilon\in\{-1,1\}^{2n+2} _{+,1}$, we can utilize \eqref{DCT05g1} and the fact
that $\langle f,f\rangle=1$ to obtain:
\[\big\langle \Phi, A^{\varepsilon(1)}(f)\ldots A^{\varepsilon(2n)} (f)A(f)A^+ (f) \Phi\big\rangle
=\big\langle \Phi, A^{\varepsilon(1)}(f)\ldots A^{\varepsilon(2n)} (f)\Phi\big\rangle
\]
Moreover, for any $\varepsilon\in\{-1,1\}^{2n+2}_{+, 1}$, by defining $\varepsilon'$ as restriction of $\varepsilon$ to the set $\{1,\ldots, 2n\}$, it is evident that $\varepsilon'$ ranges over $\{-1,1\}^{2n} _{+}$ as $\varepsilon$ running over $\{-1,1\}^{2n+2} _{+,1}$. This leads to the first equality in \eqref{DCT50b} as follows:
\begin{align*}
P_{n+1,1}=&\sum_ {\varepsilon \in\{-1,1\}^{2n+2} _{+,1} }\big\langle\Phi, A^{\varepsilon(1)}(f) \ldots A^{\varepsilon(2n)}(f)A(f)A^+(f)\Phi\big\rangle \\
=&\sum_ {\varepsilon' \in\{-1,1\}^{2n} _{+} }
\big\langle\Phi, A^{\varepsilon'(1)}(f)\ldots A^{ \varepsilon'(2n)}(f)\Phi\big\rangle=P_n
\end{align*}

Similarly, for $k=2$ and any $\varepsilon\in\{-1,1\} ^{2n+2}_{+,2}$, we can apply \eqref{DCT05g1} and the fact that $\langle f,f\rangle=1$ to get
\begin{align*}&\big\langle \Phi, A^{\varepsilon(1)}(f)\ldots A^{\varepsilon(2n-1)} (f)A(f)\big(A^+ (f)\big)^2 \Phi\big\rangle\\
=&(1+q)\big\langle \Phi, A^{\varepsilon(1)}(f)\ldots A^{\varepsilon(2n-1)} (f)A^+ (f)\Phi\big\rangle
\end{align*}
Furthermore, while $\varepsilon$ varies across $\{-1,1\}^{2n+2} _{+,2}$, the resulting $\varepsilon':= (\varepsilon(1),\ldots,\varepsilon(2n-1),1)$ runs over $\{-1,1\}^{2n} _{+}$. Therefore, we obtain the second equality in \eqref{DCT50b} in the following manner:
\begin{align*}
P_{n+1,2}=&\sum_ {\varepsilon \in\{-1,1\}^{2n+2} _{+,2} }\big\langle\Phi, A^{\varepsilon(1)}(f) \ldots A^{\varepsilon(2n-1)}(f)A(f)\big(A^+ (f)\big)^2 \Phi\big\rangle \\
=&(1+q)\sum_ {\varepsilon' \in\{-1,1\}^{2n} _{+} }
\big\langle\Phi, A^{\varepsilon'(1)}(f)\ldots A^{ \varepsilon'(2n)}(f)\Phi\big\rangle=(1+q)P_n
\end{align*}  \end{proof}

\begin{proposition}\label{DCT51}Using the notations and conventions established above, we can express the following equality:
\begin{align}\label{DCT51a}
P_{n+1,k+1}=\sum_{h=k}^nP_{n,h},\quad \text{for any } n\ge2\text{ and }k\in\{2,\ldots,n\}
\end{align}
In particular,
\begin{align}\label{DCT51a1}
P_{n+1,n+1}=P_{2,2}=1+q,\quad \text{for any } n\in{\mathbb N}^*
\end{align}
\end{proposition}
\begin{proof}The second equality in \eqref{DCT51a1} coincides with \eqref{DCT41a6}. Moreover, once \eqref{DCT51a} is proven, we obtain first equality in \eqref{DCT51a1} because
\[P_{n+1,n+1}=\sum_{h=n}^nP_{n,h}=P_{n,n}, \quad\text{for any } n\ge2
\]

Now let's proceed to prove \eqref{DCT51a}. For any $n\ge2$, $k\in\{2,\ldots,n\}$ and $\varepsilon\in\{-1,1\} ^{2n+2}_{+,k+1}$, for any $f\in \mathcal{H}^{(1)}$, we have the following:
\begin{align*}
&\big\langle \Phi, A^{\varepsilon(1)}(f)\ldots A^{\varepsilon(2n+2)} (f) \Phi\big\rangle\notag\\
=&\big\langle \Phi, A^{\varepsilon(1)}(f)\ldots A^{\varepsilon(2n-k)} (f) A(f)\big( A^+ (f)\big)^{k+1}\Phi\big\rangle\notag\\
\overset{\eqref{DCT05g1}}=&
\big\langle \Phi, A^{\varepsilon(1)}(f)\ldots A^{\varepsilon(2n-k)} (f) \big( A^+ (f)\big)^{k}\Phi\big\rangle
\end{align*}
Notice that for any $2\le k\le n$, as $\varepsilon$ running over $\{-1,1\}^{2n+2} _{+,k+1}$, the corresponding $\varepsilon' \in\{-1,1\}^{2n}$ defined by
\[\varepsilon'(h):=\begin{cases}
\varepsilon(h), &\text{ if }h\le 2n-k\\ 1,  &\text{ if }2n-k<h\le 2n
\end{cases}
\]
runs over $\cup_{h=k}^n\{-1,1\}^{2n} _{+,h}$. So,
\begin{align*}
P_{n+1,k+1}=&\sum_{\varepsilon\in\{-1,1\}^{2n+2}_{+,k+1}} \big\langle \Phi, A^{\varepsilon(1)}(f)\ldots A^{\varepsilon(2n+2)} (f) \Phi\big\rangle  \\
=&\sum_{h=k}^n\sum_{\varepsilon'\in\{-1,1\}^{2n} _{+,h} } \big\langle \Phi,A^{\varepsilon'(1)}(f)\ldots A^{\varepsilon'(2n)} (f)  \Phi\big\rangle= \sum_{h=k}^nP_{n,h}
\end{align*}\end{proof}

In summary, by denoting
\[x_{n,k}=P_{n+1,k+1},\quad\text{for any } n\in \mathbb{N}^*\text{ and } k\in\{1,\ldots,n\}
\]
\eqref{DCT51a} tells us that $\{x_{n,k}\}_{n\in \mathbb{N}^*, k\in\{1, \ldots,n\} }$ satisfies the Catalan's triangle system:
\begin{align*}
x_{n+1,k+1}=&P_{n+2,k+2}
=\sum_{h=k+1}^{n+1}P_{n+1,h}\notag\\
\overset{j:=h-1}=&\sum_{j=k}^{n}P_{n+1,j+1}
=\sum_{j=k}^{n} x_{n,j},\quad \text{for any } n\in \mathbb{N}^* \text{ and  }k\in\{1,\ldots,n\}
\end{align*}
However, the {\it boundary conditions} are quite different from \eqref{DCT-CTSbou}. In fact, it follows from \eqref{DCT51a1} and \eqref{DCT50a} that
\begin{equation}\label{DCT41a9}
x_{n,n}=x_{1,1}=1+q,\quad x_{n,1}=(1+q)P_n
, \quad \text{ for any } n\in\mathbb{N}^*
\end{equation}

As the end of this section, it's worth noting that the explicate
expressions of $P_n$'s can be found in \cite{YGLu2022d}.

\end{document}